\documentclass[12pt,onecolumn]{IEEEtran}

\usepackage{graphicx,color}
\usepackage{amsmath,amssymb,mathrsfs,amsthm}
\usepackage[vlined,ruled]{algorithm2e}
\SetNlSty{textbf}{}{:}
\usepackage{subfigure}
\IEEEoverridecommandlockouts
\title{On cooperative patrolling: optimal trajectories, complexity
  analysis, and approximation algorithms}

\author{Fabio Pasqualetti, Antonio Franchi, and Francesco
  Bullo
  \thanks{This material is based upon work supported in part by NSF
    grant IIS-0904501 and CPS-1035917. The authors thank the reviewers
    for their thoughtful and constructive remarks.}
  \thanks{Fabio Pasqualetti and Francesco Bullo are with the Center for
    Control, Dynamical Systems and Computation, University of California at
    Santa Barbara, {\tt \{fabiopas,bullo\}@engineering.ucsb.edu}}%
  \thanks{Antonio Franchi is with the Department of Human Perception,
    Cognition and Action, Max Plank Institute for Biological Cybernetics,
    {\tt antonio.franchi@tuebingen.mpg.de}}%
}

\def\figCommChain{\centering\includegraphics[width=0.45\columnwidth
]{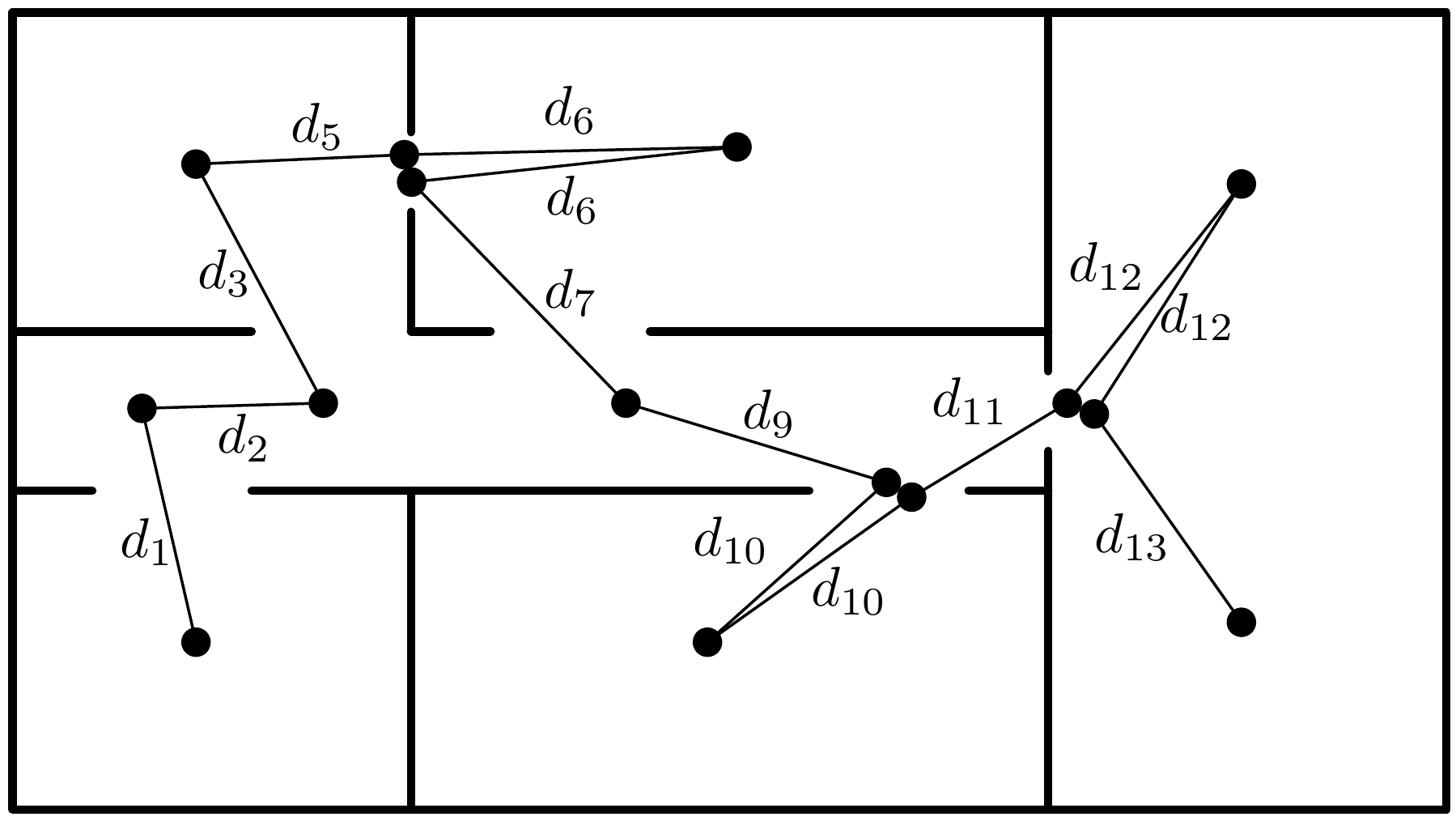}}

\newtheorem{theorem}{Theorem}[section]
\newtheorem{lemma}[theorem]{Lemma}
\newtheorem{proposition}[theorem]{Proposition}

\newtheorem{problem}{Problem}
\newtheorem{remark}{Remark}
\newtheorem{example}{Example}

\newcommand{\setdef}[2]{\{#1 \; | \; #2\}}
\newcommand{\subscr}[2]{{#1}_{\textup{#2}}}
\newcommand{\until}[1]{\{1,\dots,#1\}}

\renewcommand{\theenumi}{(\roman{enumi})}

\SetAlgorithmName{Trajectory}{ of Trajectories}{Trajectories}

\begin{document}
\maketitle

\begin{abstract}
  The subject of this work is the patrolling of an environment with
  the aid of a team of autonomous agents. We consider both the design
  of open-loop trajectories with optimal properties, and of
  distributed control laws converging to optimal trajectories. As
  performance criteria, the \emph{refresh time} and the \emph{latency}
  are considered, i.e., respectively, time gap between any two visits
  of the same region, and the time necessary to inform every agent
  about an event occurred in the environment. We associate a graph
  with the environment, and we study separately the case of a chain,
  tree, and cyclic graph. For the case of chain graph, we first
  describe a minimum refresh time and latency team trajectory, and we
  propose a polynomial time algorithm for its computation. Then, we
  describe a distributed procedure that steers the robots toward an
  optimal trajectory. For the case of tree graph, a polynomial time
  algorithm is developed for the minimum refresh time problem, under
  the technical assumption of a constant number of robots involved in
  the patrolling task. Finally, we show that the design of a minimum
  refresh time trajectory for a cyclic graph is \emph{NP-hard}, and we
  develop a constant factor approximation algorithm.
\end{abstract}

\section{Introduction}
The recent development in the autonomy and the capabilities of mobile
robots greatly increases the number of applications suitable for a
team of autonomous agents. Particular interest has been received by
those tasks requiring continual execution, such as the monitoring of
oil spills \cite{JC-RF:07}, the detection of forest fires
\cite{DBK-RWB-RSH:08}, the track of border changes
\cite{SS-SM-FB:06f}, and the patrol (surveillance) of an environment
\cite{YE-AS-GAK:08}. The surveillance of an area of interest requires
the robots to continuously and repeatedly travel the environment, and
the challenging problem consists in scheduling the robots trajectories
so as to optimize a certain performance criteria. The reader familiar
with network location, multiple traveling salesman, or graph
exploration problems may observe a close connection with the
patrolling problem we address, e.g., see
\cite{BCT-RLF-TJL:83,TB:06,EMA-RH-AL:06}. It is worth noting, however,
that these classical optimization problems do not capture the
repetitive, and hence dynamic, aspect of the patrolling problem, nor
the synchronization issues that arise when a timing among the visits
of certain zones is required.

A precise formulation of the patrolling problem requires the
characterization of the robots capabilities, of the environment to be
patrolled, and of the performance criteria. In this work, we assume
the robots to be identical and capable of sensing and communicating
within a certain spatial range, and of moving according to a first
order integrator dynamics with bounded speed. We represent the
environment as a graph, in which the vertices correspond to physical
and strategically important locations, and in which the edges denote
the possibility of moving and communicating between locations. We
assume that, when a robot is placed at each of the graph vertices, the
union of the sensor footprints provides complete sensor coverage of
the environment. Regarding the performance criteria of a patrolling
trajectory, we consider (i) the time gap between any two visits of the
same region, called {\it refresh time}, and (ii) the time needed to
inform the team of robots about an event occurred in the environment,
called {\it latency}. Loosely speaking, refresh time and latency
reflect the effectiveness of a patrolling team in detecting events in
the environment and in organizing remedial actions. For both the
refresh time and latency optimization problem, we focus on the worst
case analysis, even though the average refresh time and the average
latency cases are also of interest. Notice that for the latency to be
finite, the motion of the robots needs to be synchronized. For
instance, if two robots are allowed to communicate only when they
simultaneously occupy two adjacent vertices of the graph, then they
need to visit those vertices at the same time in a finite latency
trajectory.

The patrolling problem is receiving increasing attention because of
its fundamental importance in many security applications, e.g., see
\cite{MB-GSS:07,NA-SK-GAK:08,FA-NB-NG:09,AM-LP-GA-FC:09,DA-PO-XH:10}. Although
many solutions have been proposed, the problem of designing minimum
refresh time and latency team trajectories for a general environment
is, to date, an open problem.
Almost all traditional approaches rely on space decomposition, and
traveling salesperson tour computation
\cite{AA-GC-HS-PT-TM-VC-YC:04}. In \cite{AM-GR-JDZ-AD:03} an empirical
evaluation of existing patrolling heuristics is performed. In
\cite{YC:04} two classes of strategies are presented, namely the
cyclic- and the partition-based strategy. In the cyclic-based
strategy, the robots compute a closed route through the viewpoints,
and travel repeatedly such route at maximum speed. Clearly, in the
case of a single robot, if the tour is the shortest possible, then the
cyclic-based strategy performs optimally with respect to the refresh
time and latency criteria. In the partition-based strategy, the
viewpoints are partitioned into $m$ subsets, being $m$ cardinality of
the team, and each robot is responsible for a different set of
viewpoints. To be more precise, each robot computes a closed tour
visiting the viewpoints it is responsible for, and it repeatedly moves
along such tour at maximum speed. Still in \cite{YC:04}, the two
classes of strategies are compared, and it is qualitatively shown that
cyclic-based strategies are to be preferred whenever the ratio of the
longest to the shortest edge of the graph describing the environment
is small, while, otherwise, partition-based policies exhibit better
performance. In \cite{YE-AS-GAK:08} and \cite{DBK-RWB-RSH:08}, an
efficient and distributed solution to the perimeter patrolling problem
for robots with zero communication range is proposed. By means of some
graph partitioning and graph routing techniques, we extend the results
along these directions, e.g., by considering the case of a nonzero
communication range for the perimeter patrolling, and by
characterizing optimal strategies for different environment
topologies. An important variant of the patrolling problem is known as
persistent surveillance, e.g., see \cite{SLS-DR:10}. Differently to
our setup, a dynamically changing environment is considered for the
persistent surveillance problem, and performance guarantees are
offered only under a certain assumption on the rate of change of the
regions to be visited.

It is worth mentioning that a different approach to the design of
patrolling trajectories relies on the use of pebbles or chemical
traces to mark visited regions, e.g., see
\cite{REK:88,GD-MJ-EM-DW:78,MAB-AF-DR-AS-SV:02,IAW-ML-AMB:98}. These
techniques, although effective even without a global representation of
the environment, and with severe communication constraints, do not
explicitly deal with the optimality of the patrolling trajectories,
and they represent therefore a complementary area of research with
respect to this work.

The main contributions of this work are as follows. We introduce and
mathematically formalize the concept of refresh time and latency of a
team trajectory, and we formally state the patrolling optimization
problem. We propose a procedure to build a graph (roadmap) to
represent the topological structure of the area to be patrolled, and
we study separately the case of a chain, tree, and cyclic (not
acyclic) graph. We exhaustively discuss the case of a chain
roadmap. First, we characterize a family of minimum refresh time and
latency team trajectories, which can be computed by optimally
partitioning the chain graph among the robots. Second, we derive a
centralized polynomial time algorithm to compute an optimal partition,
and, ultimately, to design an optimal team trajectory. Our
partitioning procedure is based upon a bisection method, and it is
also amenable to distributed implementation. Third, we develop a
distributed procedure for the robots to converge and synchronize along
an optimal trajectory, so as to minimize the refresh time and latency
criteria. Fourth and finally, we test the robustness of our methods
through a simulation study. When the roadmap has a tree or cyclic
structure, we focus on the refresh time optimization problem, and we
do not consider the latency optimization nor the design of distributed
algorithms. For the case of a tree roadmap, we reduce the minimum
refresh time patrolling problem to a known graph optimization
problem. We show that the computational complexity of the minimum
refresh time patrolling problem is polynomial in the number of
vertices of the roadmap, and, under the assumption of a fixed and
finite number of robots,
we identify a polynomial time centralized algorithm to
compute a minimum refresh time team trajectory. For the case of a
cyclic roadmap, we show that the patrolling problem is an
\emph{NP-hard} optimization problem. We propose two approximation
algorithms, and we characterize their performance. The first
approximate solution is extremely easy to compute, but its performance
depends upon the ratio between the longest and the shortest edge in
the graph representing the environment. The second approximation
algorithm is based on a polynomial time path-covering procedure, and
it allows us to compute a team trajectory whose refresh time is within
a factor of $8$ from the minimum refresh time for the given
environment (cf. Fig. \ref{fig:path_cover} for an example). To the
best of our knowledge, this algorithm is the first constant factor
approximation algorithm for the \textit{NP-hard} minimum refresh time
patrolling problem.

A preliminary version of this work appeared in
\cite{FP-AF-FB:10e}. With respect to the latter manuscripts, in this
current work we introduce and solve the latency optimization problem,
we perform a numerical study to analyze the robustness of our
algorithmic procedures, and we improve the presentation of the results
on the refresh time optimization problem.

The rest of the paper is organized as follows. The notation and the
problem under consideration are in Section \ref{sec:prob_set}, where
we also show that the patrolling problem is, generally,
computationally hard. Section \ref{sec:refr_time}, \ref{sec:latency},
and \ref{sec:distr_alg} contain our results for the patrolling of a
chain environment. We characterize a minimum refresh time and latency
team trajectory, and we derive a centralized and a decentralized
algorithm for its computation. In Section \ref{sec:simulations} we
perform a simulation study to show some robustness and
reconfigurability properties of our distributed procedure. Section
\ref{sec:tree_patrol} contains our results for the patrolling of a
tree environment. We describe a minimum refresh time team trajectory
on a tree roadmap, and we characterize the complexity of computing an
optimal solution. Section \ref{sec:heurist} deals with the general
case of cyclic environment, and it contains our approximation
procedures. Our conclusion and final remarks are in Section
\ref{sec:future_work}.

\section{Robotic model and preliminary concepts}\label{sec:prob_set}
\subsection{Robots on roadmaps with sensor coverage and communication
  connectivity}
We will be using the standard motion planning notation, and we refer
the reader to \cite{SML:06} for a comprehensive treatment of the
subject. We are given a team of $m>0$ identical robots, capable of
sensing, communicating, and moving in a connected environment. We make
the following combined assumptions on the robot capabilities and on
the environment to be patrolled.

Regarding sensing, we assume that the environment can be completely covered
by simultaneously placing a robot at each of a set of $n>m$
\emph{viewpoints} in the configuration space. In other words, if $m=n$
robots were available and placed at the $n$ viewpoints, then the union of
the sensors footprint of the robots would provide complete sensor coverage
of the environment. We assume that each viewpoint is required for complete
sensor coverage. Finally, we assume $n>m$ so that at least one robot needs
to visit more viewpoints for the entire environment to be monitored over
time.

Regarding communication, we associate an undirected graph $G$ with the
environment, whose vertices are the $n$ viewpoints, and in which there is
an edge between two vertices if two robots placed at those viewpoints are
able to communicate to each other. We assume that $G$ is connected.
In what follows we design cooperative patrolling algorithms with sporadic
communication, in the sense that two robots are required to communicate
only when they occupy adjacent vertices. The occurrence of additional
communication links can be easily incorporated into the algorithms and
cannot decrease their performance.

Regarding motion, we assume that the robots are holonomic, i.e., they are
modeled as first order integrators and move at most at unit speed.
Additionally, we turn the graph $G$ into a robotic roadmap~\cite{SML:06}
and a metric weighted graph as follows: to each pair of viewpoints that are
neighbors in $G$, we associate a unique path connecting them. We select
these paths so that the set of path lengths, adopted as edge weights,
verify the triangle inequality. (For example, the shortest paths between
viewpoints constitute a suitable choice.)
%
We assume that each robot remains always on the roadmap.

In summary, the combined assumptions on the robot capabilities and on the
environment are that: the vertices of $G$ provide complete sensor coverage
of the environment and each edge of $G$ corresponds to both a communication
edge and a motion path.  Hence, we refer to $G$ as a roadmap with sensor
coverage and communication connectivity.



\subsection{On suitable roadmaps and robots capabilities}
The problem of constructing a roadmap from an environment is here
discussed.


\begin{example}[Roadmap computation with omnidirectional sensing and
  communication in known environments]
  Assume that the robots are holonomic vehicles moving at bounded speed,
  and equipped with an omnidirectional sensing device, and a line-of-sight
  communication device. If a map of the environment is available, then a
  valid roadmap is obtained by solving an Art Gallery Problem with
  Connectivity \cite{JOR:87}.  A solution to the Art Gallery Problem with
  Connectivity is a set of locations, called guards, with the following two
  properties: each point in the environment is visible by at least one
  guard and the visibility graph of the guards is connected. An example
  roadmap is given in Fig.~\ref{fig:CommGraph}. (A distributed sensor-based
  algorithm for the Art Gallery Problem with Connectivity is given
  in~\cite{KJO-AG-FB:08w}.)
\end{example}

\begin{figure}[]
  \centering
  \includegraphics[width=0.4\columnwidth]{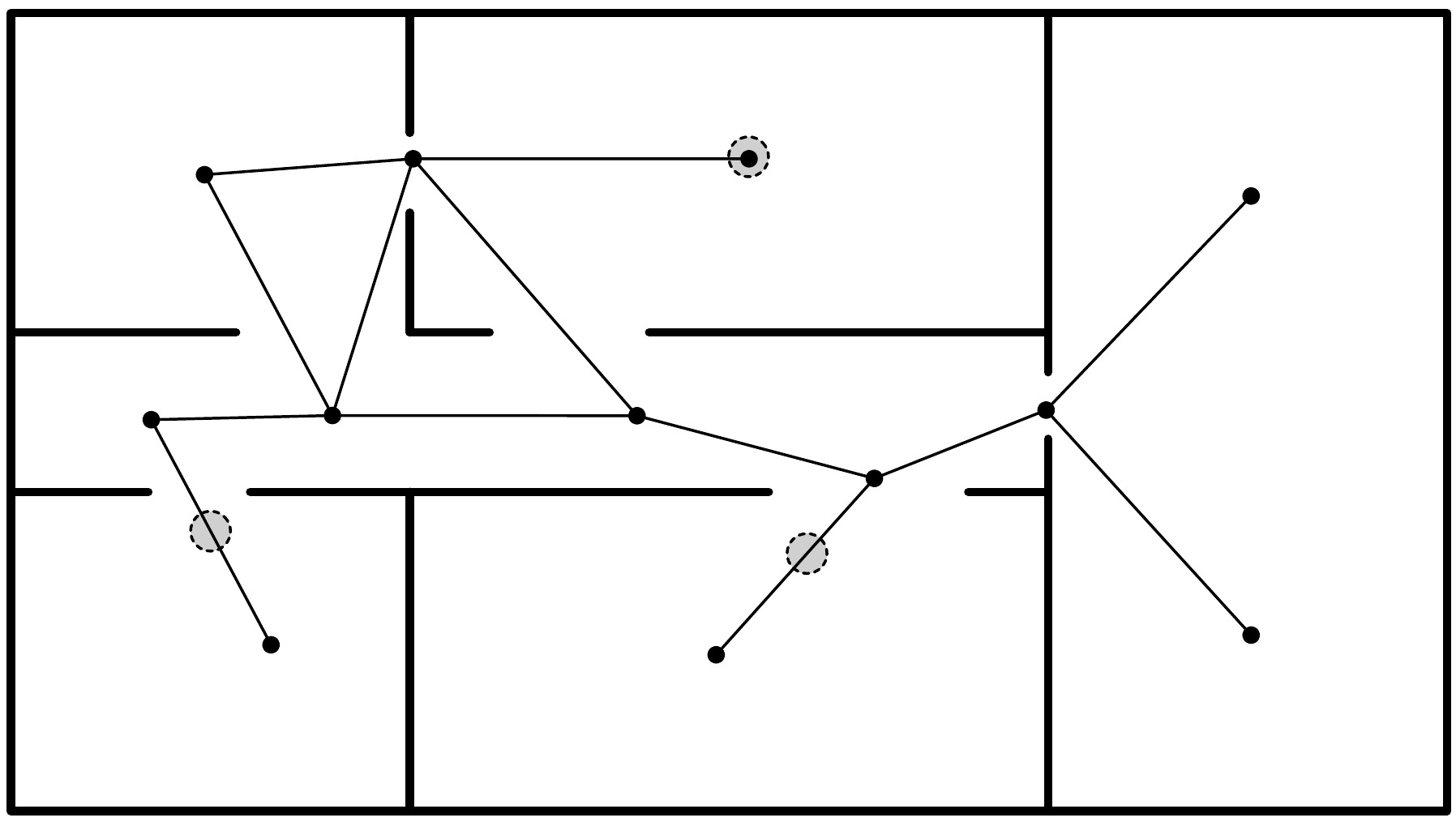}\\
  \caption{A polygonal environment and its associated roadmap. The
    viewpoints are chosen in a way that the environment is completely
    covered by placing a robot at each of the $12$ viewpoints. The edges of
    the roadmap denote the possibility for the $3$ robots of both moving
    and communicating between pair of connected viewpoints. The weight of
    each edge corresponds to its length.}\label{fig:CommGraph}
\end{figure}

\begin{example}[Roadmap  computation  for  general  robots  in  an  unknown
  environment] %
  If the  environment is unknown,  then an exploration algorithm  should be
  used   to  obtain   a  representation   of  the   environment,   see  for
  example~\cite{GD-MJ-EM-DW:78,IAW-ML-AMB:00,AF-LF-GO-MV:09}.    While   or
  after  exploring the  environment, a  robot should  select  viewpoints to
  provide full sensor  coverage.  By construction, the existence  of a path
  between  two viewpoints  is automatically  guaranteed by  the exploration
  algorithm.   Moreover,   if  communication  between   two  path-connected
  viewpoints $(v_i,v_j)$ is not guaranteed, then the graph may be augmented
  with  additional viewpoints along  the path  connecting $v_i$  and $v_j$.
  The  roadmap resulting  from  these steps  features  sensor coverage  and
  communication connectivity.
\end{example}


The following remarks are in order. First, a roadmap representing the
environment is in general not unique, and the patrolling performance
depends upon the particular choice. In this work, we do not address
the problem of choosing the roadmap that yields optimal
performance. Instead, we present efficient algorithms for the
patrolling problem, which can also be used to compare different
roadmaps on the basis of the corresponding patrolling
performance. Second, for the implementation of our algorithms, a robot
does not need to travel exactly along the roadmap. Indeed, a robot
only needs to arrive sufficiently close to a viewpoint, or to be able
to circle around it to provide sensor coverage. A related example is
in Section \ref{sec:simulations}, where the arrival delay of a robot
can be interpreted as the uncertainty in the motion of the
robots. Third, the global knowledge of the roadmap may not be required
for the execution of a patrolling trajectory. Indeed, in general, each
robot only visits a subpart of the roadmap. Fourth and finally,
collisions are prevented by letting two robots exchange their roles
every time they are about to collide. Indeed, since no robot is
assigned to a specific viewpoint, if robot $i$ is about to collide
with robot $j$ at time $T$, then, without affecting the performance of
the team trajectory, collision is avoided by redefining the $i$-th and
$j$-th trajectory as $\bar x_i (t) = x_j (t)$ and $\bar x_j (t) = x_i
(t)$ for $t\ge T$; see the notion of order invariant trajectory
below. Communication or sensing devices can be used to detect and
prevent possible collisions.

\subsection{Refresh time of team trajectories: definitions and a
  preliminary complexity result}
For a team of $m$ robots with specific capabilities, let $G=(V,E)$ be
a robotic roadmap with sensor coverage and communication
connectivity. A \emph{team trajectory} $X$ is an array of $m$
continuous and piecewise-differentiable trajectories $x_1
(t),\dots,x_m (t)$ defined by the motion of the robots on the roadmap
$G$, i.e., $x_i: [0,\subscr{T}{f}] \mapsto G$, for $i \in \until{m}$,
where $\subscr{T}{f} \in \mathbb{R}$ is a time horizon of interest,
much larger than the time required by a single robot to visit all
viewpoints. We say that a viewpoint $v\in V$ is visited at time $t$ by
robot $i$ if $x_i(t)=v$. We define the \emph{refresh time} of a team
trajectory $X$, in short $\text{RT}(X)$, as the longest time interval
between any two consecutive visits of any viewpoint, i.e.,
\begin{align*}
  \text{RT}(X)=\max_{v \in V} \max_{(t_1,t_2) \in \Omega(v,X)}t_2 -t_1
\end{align*}
where $\Omega(v,X) = \setdef{ (t_1,t_2) \in [0,\subscr{T}{f}]^2, t_1
  \le t_2, }{x_i(t)\neq v, \forall i\in\until{m}, t_1 < t < t_2}$.

\begin{remark}[Existence of a minimum]
  We claim that there exists a team trajectory with minimum refresh
  time and prove it as follows. Without loss of generality, we
  restrict our attention to team trajectories in which each robot
  moves at maximum speed along the edges and stops for certain
  durations, possibly equal to zero, at the viewpoints. Thus, a team
  trajectory can be described as a tuple of the form $(S,\Delta) =
  \{(S_1,\Delta_1),\dots,(S_m,\Delta_m)\}$, where $S_i$ contains the
  sequence of viewpoints visited by robot $i$, and $\Delta_i$ contains
  the waiting times at the visited vertices. Notice that the time
  horizon $\subscr{T}{f}$ is finite, the length of each edge is
  positive, the number of vertices is finite, and the speed of the
  robots is bounded. It follows that the length of each sequence $S_i$
  is finite, and, therefore, each $S_i$ takes value in a finite set.
  Now, for each possible sequence of visited vertices, the refresh
  time is a continuous function of only the waiting times, and each
  waiting time lies in the compact interval $[0,\subscr{T}{f}]$.
  Because any continuous function defined over a compact region admits
  a point of minimum value, we conclude that there exists a team
  trajectory with minimum refresh time.
\end{remark}

\begin{problem}[Team refresh time]\label{prob:minRT}
  Given a roadmap and a team of robots, find a minimum refresh time
  team trajectory.
\end{problem}



In Section \ref{sec:latency} we present a different optimization
problem, which deals with the possibility for a robot to communicate,
possibly with multiple hops, with every other robot in the team. We
now conclude this section with our first result on the computational
complexity of the Team refresh time problem. For a detailed discussion
of the main aspects of the computational complexity theory, we refer
the interested reader to \cite{MRG-DSJ:79}. Recall that an
optimization problem is said to be \emph{NP-hard} if it is,
informally, as hard as the hardest problem in the class \emph{NP}, for
which no polynomial time algorithm is known to compute an optimal
solution.


\begin{theorem}[Computational complexity]\label{thm:complexity}
  The \emph{Team refresh time} problem is \textit{NP-hard}.
\end{theorem}
\begin{proof}
  This statement can be shown by reduction from the Traveling Salesman
  problem \cite{MRG-DSJ:79}. In fact, if $m = 1$, since the speed of
  the robots is bounded, then a minimum refresh time trajectory
  consists of moving the robot at maximum speed along a shortest
  closed tour visiting the viewpoints. The problem of finding a
  shortest tour through a set of points in the plane, also known as
  Traveling salesman problem, is an \textit{NP-hard} problem
  \cite{MRG-DSJ:79}. Hence, by restriction, the Team refresh time
  problem is also \textit{NP-hard}.
\end{proof}

Following Theorem~\ref{thm:complexity}, the minimum refresh time
optimization problem is generally computationally hard. In this work,
we first identify two roadmap structures for which there exists an
efficient solution to the Team refresh time problem, and then we
describe two approximation algorithms to deal with the general case.


\section{Minimum refresh time team trajectory on a chain
  roadmap}\label{sec:refr_time}
We characterize in this section an optimal refresh time team
trajectory when the roadmap associated with the environment has a
chain structure.

\subsection{Open loop team trajectory characterization}
Let $N_i$ denote the neighbor set of the vertex $i$, and let $|N_i|$
denote the degree of $i$, i.e., the cardinality of the set $N_i$. A
chain roadmap is an undirected, connected, and acyclic roadmap, in
which every vertex has degree two, except for two vertices which have
degree one. Without losing generality, we assume that the $n$ vertices
are ordered in a way that $N_1= \{2\}$, $N_n = \{n-1\}$, and
$N_i=\{i-1,i+1\}$ for each $i\in \{2,\dots,n-1\}$. We define a
relative order of the robots according to their position on the
roadmap. A team trajectory is \emph{order invariant} if the order of
the robots does not change with time, i.e., if $x_i(t) \le x_{i+1}(t)$
for each $i\in\{1,\ldots,m-1\}$ and for every instant $t \in
[0,\subscr{T}{f}]$, where $x_i (t)$ denotes the distance at time $t$
on the roadmap from the first vertex of the chain to the position of
the $i$-th robot.

\begin{proposition}[Order invariant team
  trajectory]\label{prop:eqv_traj}
  Let $X$ be a team trajectory. There exists an order invariant team
  trajectory $\bar X$ such that $\text{RT}(X)=\text{RT}(\bar X)$.
\end{proposition}
\begin{proof}
  Let $X$ be a team trajectory, and consider the permutation matrix
  $P(t)$, that keeps track of the order of the robots at time $t$,
  i.e., such that the $(i,j)$-th entry of $P(t)$ is $1$ if, at time
  $t$, the $i$-th robot occupies the $j$-th position in the chain of
  robots, and it is $0$ otherwise. Since $X$ is continuous, anytime
  the function $P(t)$ is discontinuous, the positions of the robots
  directly involved in the permutation overlap. Therefore, the order
  invariant team trajectory $\bar X=P^{-1}(t) X(t)$ is a feasible team
  trajectory, and it holds $\text{RT}(\bar X)=\text{RT}(X)$.
\end{proof}

Let $V_i\subseteq V$ be the set of viewpoints visited over time by the
agent $i$ with the trajectory $x_i$, and let the \emph{image} of the
team trajectory $X$ be the set $\{V_1,\dots,V_m\}$. Notice that
different team trajectories may have the same image. Throughout the
paper, let $l_i = \min_{v \in V_i} v$, $r_i = \max_{v \in V_i} v$, and
$d_i = r_i - l_i$. Finally, let $\text{RT}^*=\min_X \text{RT}(X)$. A
team trajectory is \textit{non-overlapping} if $V_i \cap V_j
=\emptyset$ for all $i \neq j$.

\begin{proposition}[Non-overlapping team
  trajectory]\label{prop:disjoint_sets}
  Given a chain roadmap, there exists an order invariant and
  non-overlapping team trajectory with refresh time $\text{RT}^*$.
\end{proposition}
\begin{proof}
  Let $X^*$ be a minimum refresh time team trajectory, and let $X$ be
  the order invariant team trajectory obtained from $X^*$ as in
  Proposition \ref{prop:eqv_traj}. Clearly
  $\text{RT}(X)=\text{RT}^*$. Let $\{V_1,\dots,V_m\}$ be the image of
  $X$, and note that $V=\cup_{i=1}^{m}V_i$. Consider the partition of
  $V$ defined as
  \begin{align*}
    \bar{V}_1&=V_1,\\
    \bar{V}_i&=V_i \setminus \cup_{j=1}^{i-1}V_j, \quad i\in \{2,\dots,m\}.
  \end{align*}
  For every nonempty $\bar V_i$, let $\bar l_i = \min_{v \in \bar V_i}
  v$, $\bar r_i = \max_{v \in \bar V_i} v$, and $\bar d_i = \bar r_i
  -\bar l_i$. Note that, by construction, the viewpoint $\bar l_i$ is
  visited by the robot $i$ and, possibly, by the robots $j>i$. Also,
  because $X$ is order invariant, we have $x_i(t)\le x_j(t)$. It
  follows that $\text{RT}(X) \ge 2\max_i \bar d_i$. Consider now the
  team trajectory $\bar X$ with image $\{\bar V_1, \dots, \bar V_m\}$,
  and assume that the robots sweep periodically at maximum speed their
  segment. Then $\text{RT}(\bar X) = 2\max_i \bar d_i$, so that $\bar
  X$ is an order invariant and non-overlapping team trajectory with
  minimum refresh time.
\end{proof}

Given a chain graph on the viewpoints $V$, let
$\Pi_m=\{\pi_1,\dots,\pi_m\}$ be an \textit{$m$-partition} of $V$,
i.e., $\pi_1,\dots,\pi_m$ is a collection of subsets of $V$ such that
$\pi_i \cap \pi_j = \emptyset$ whenever $i \neq j$, and $V =
\bigcup_{i=1}^m \pi_i$. Additionally, let the dimension of the
partition $\Pi_m$ equal the longest distance between any two
viewpoints in the same cluster, i.e., $\operatorname{dim}(\Pi_m) =
\max_{i\in \{1,\dots,m\}} \left( \max_{v \in \pi_i} v - \min_{v \in
    \pi_i} v \right)$, where $\max_{v \in \pi_i} v - \min_{v \in
  \pi_i} v = 0$ if $\pi_i = \emptyset$. Following Proposition
\ref{prop:disjoint_sets}, there exists a minimum refresh time team
trajectory whose image coincide with an $m$-partition of $V$. We now
show that the minimum refresh time equals twice the dimension of an
optimal $m$-partition.

\begin{theorem}[Minimum refresh time]\label{thm:min_refre_time}
  Let $G$ be a chain roadmap, and let $m$ be the number of
  robots. Then $\text{RT}^*=2 \min_{\Pi_m} \operatorname{dim}(\Pi_m)$.
\end{theorem}
\begin{proof}
  As a consequence of Propositions \ref{prop:eqv_traj} and
  \ref{prop:disjoint_sets}, there exists a minimum refresh time team
  trajectory whose image coincides with an $m$-partition
  $\Pi_m$. Since each robot is assigned a different cluster, and the
  speed of the robots is bounded by $1$, we have $\text{RT}^* \ge 2
  \text{dim}(\Pi_m)$. Consider a team trajectory $X$ in which each
  robot continuously sweeps at maximum speed the cluster it is
  assigned to. Clearly, $\text{RT}(X) = \text{RT}^*=2
  \text{dim}(\Pi_m)$.
\end{proof}

\IncMargin{.3em}
\begin{algorithm}[t]
  {\footnotesize
   \SetKwInOut{Input}{Input}
   \SetKwInOut{Set}{Set}
   \SetKwInOut{Title}{Algorithm}
   \SetKwInOut{Require}{Require}
   \Input{$l_i := \min_{v \in V_i} v$, $r_i := \max_{v \in V_i} v$,
     $d_i :=r_i-l_i$\;}
   \Require{an optimal partition of the chain graph\;}
   
   \BlankLine
  
    \nl$x_i(t):=l_i$ for $t:=0,2d_i,4d_i,\dots$\;
    \nl$x_i(t):=r_i$ for $t:=d_i,3d_i,5d_i,\dots$\;

  \caption{\textit{Minimum refresh time
      trajectory on a chain roadmap ($i$-th robot)}}
  \label{algo:max_sweep}}
\end{algorithm} 
\DecMargin{.3em}

We have shown that a minimum refresh time trajectory consists of
letting the robot sweep at maximum speed a part of the chain
graph. Such a trajectory is more formally described for the $i$-th
robot in Trajectory \ref{algo:max_sweep}, where we only characterize
the instants of time at which robot $i$ changes its velocity vector,
and we assume that it moves at maximum speed otherwise.

\begin{remark}[Average partition]
  By removing the longest edges in the chain the average length of the
  clusters is minimized. In general, such partition does not minimize
  the dimension of the $m$-partition, and hence it is not optimal in
  our sense. An example is in Fig. \ref{fig:average}.
  \begin{figure}
    \centering
    \subfigure[]{
          \includegraphics[width=.55\columnwidth]{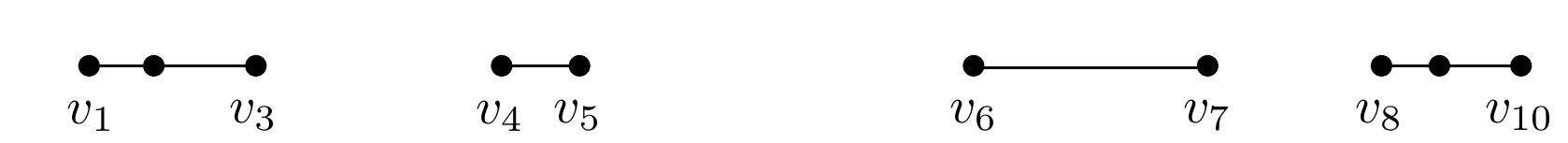}
          \label{fig:OptClustb}}\\
        \subfigure[]{
          \includegraphics[width=.55\columnwidth]{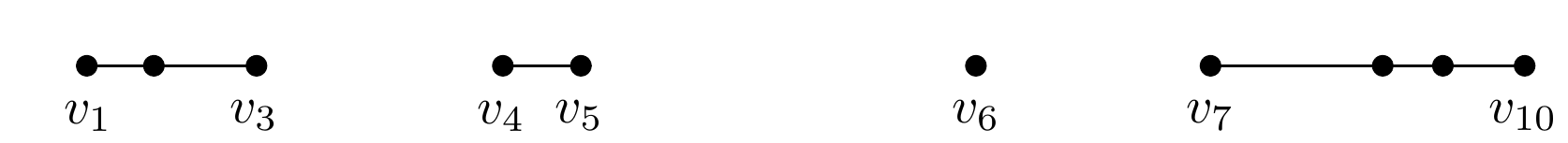}
          \label{fig:OptClustc}}\\
        \caption{A $4$-partition with minimum (maximum) dimension is
          in Fig. \ref{fig:OptClustb}. The dimension of this partition
          is $v_7-v_6$. A $4$-partition with minimum (average)
          dimension is obtained by removing $3$ longest edges and it
          is reported in Fig. \ref{fig:OptClustc}. The dimension of
          this partition is $v_{10}-v_7>v_7-v_6$.}
    \label{fig:average}
  \end{figure}
\end{remark}

\subsection{Optimal $m$-partition centralized computation}
\begin{figure}
	\centering
        \subfigure[]{
          \includegraphics[width=.55\columnwidth]{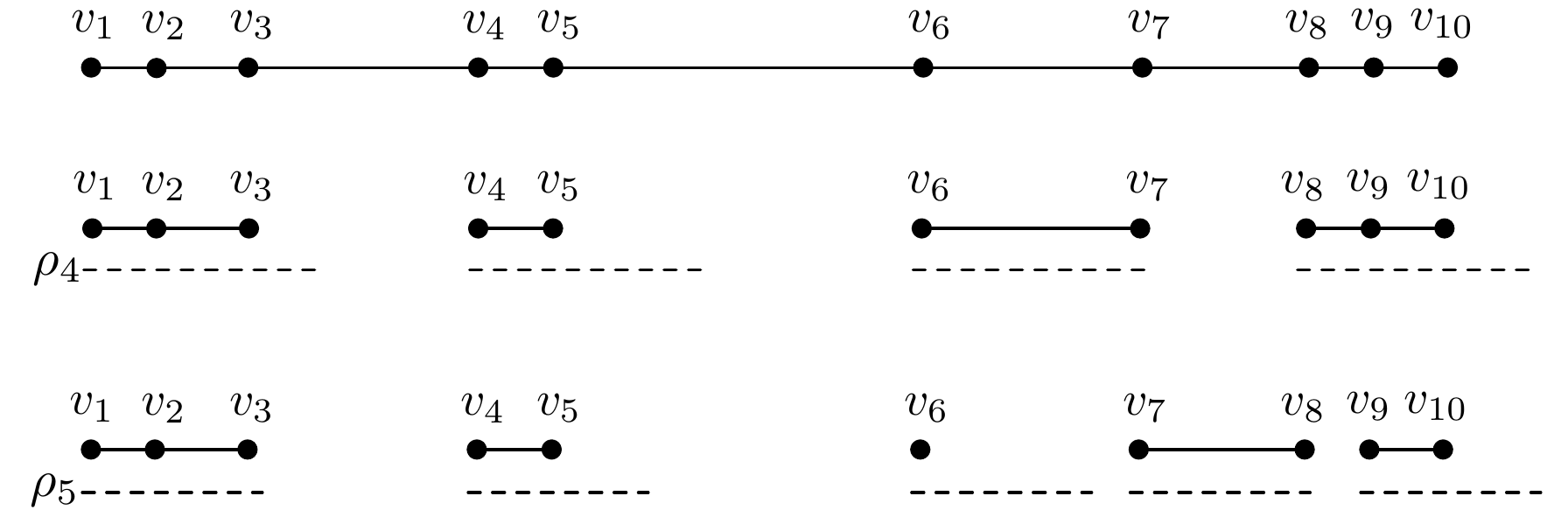}
          \label{fig:LeftInduc}}
        \subfigure[]{
          \includegraphics[width=.55\columnwidth]{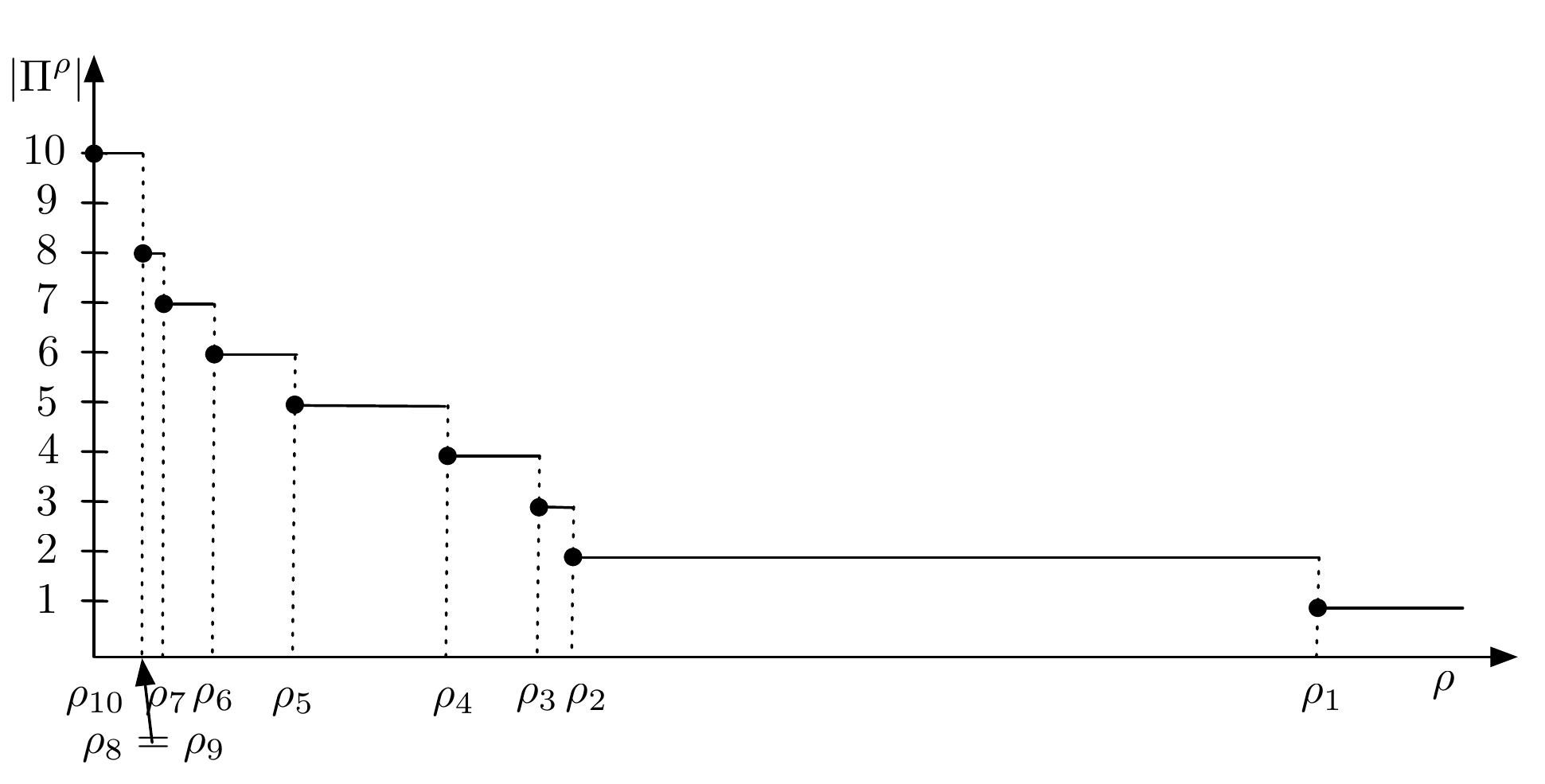}
          \label{fig:LeftInducCard}}
        \caption{In \ref{fig:LeftInduc} the left-induced partition of
          length $\rho_4$ and $\rho_5$, with $\rho_5< \rho_4$, for the
          chain roadmap with vertices $\{v_1,\dots,v_{10}\}$. The
          cardinalities are $|\Pi^{\rho_4}|=4$ and $|\Pi^{\rho_5}|=5$,
          respectively. In \ref{fig:LeftInducCard} the cardinality
          $|\Pi^\rho|$ is plotted as a function of the length
          $\rho$. Notice that, because $v_2-v_1=v_{10}-v_9$, the
          function $|\Pi^\rho|$ does not assume the value $9$.}
\end{figure}

In the remaining part of the section we describe an algorithm to
compute an optimal $m$-partition. For a set of viewpoints $V$, we call
\emph{left-induced} partition of length $\rho \in \mathbb{R}_{\ge 0}$
the partition $\Pi^\rho=\{\pi_i\}$ defined recursively as
(cf. Fig. \ref{fig:LeftInduc})
\begin{align}\label{eq:left-induced}
  \pi_i = \setdef{v\in V}{a_i\le v \le a_i+\rho},\;\; i = 1, \dots \,,
\end{align}
where
\begin{align*}
  a_1&=v_1,\\
  a_j&=\min \setdef{v\in V}{v>a_{j-1}+\rho}, \;\; j = 2, \dots \,.
\end{align*}
The cardinality $| \Pi^\rho |$ corresponds to the integer $j$ such
that \mbox{$\setdef{v \in V }{ v > a_j +\rho} =\emptyset$}.
Observe that the function $\rho \mapsto |\Pi^\rho|$ is monotone,
non-increasing, and right-continuous
(cf. Fig. \ref{fig:LeftInducCard}). Let $\{\rho_1,\dots,\rho_{n-1}\}
\in \mathbb{R}_{\ge 0}^{n-1}$ be the discontinuity points of the
function $\rho \mapsto |\Pi^\rho|$, then, for $k\in\{1,\ldots,n-1\}$,
\begin{align}
  \label{eq:disc_points}
  \begin{split} 
    |\Pi^\rho| \leq k,\quad \text{if } \rho \geq \rho_k,\\
    |\Pi^\rho| > k, \quad \text{if } \rho < \rho_k.
  \end{split}
\end{align} 
Note that two or more discontinuity points of $|\Pi^\rho|$ may
coincide, so that the function $|\Pi^\rho|$ may not assume all the
values in the set $\{1,\dots,n\}$, e.g., in
Fig. \ref{fig:LeftInducCard} the value $|\Pi^\rho|$ is never equal to
$9$.

\SetAlgorithmName{Algorithm}{ of Algorithms}{Algorithms}
\IncMargin{.3em}
\begin{algorithm}[t]
  {\footnotesize
  \SetKwData{Left}{left} 
  \SetKwData{Right}{right} 
  \SetKwData{Position}{position} 
  \SetKwInOut{Input}{Input}
  \SetKwInOut{Output}{Output}
  \SetKwInOut{Set}{Set}  
  \SetKwInOut{Title}{Algorithm}
  \Input{Viewpoints $v_1,\dots,v_n$, Number of robots $m \le n$,
    Tolerance $0 < \varepsilon < \frac{v_n}{m}$\;}
  \Set{$a:=0$, $b:=\frac{2v_n}{m}$,
    $\rho:=\frac{(a+b)}{2}$\;}
  
  \BlankLine

  \nl\While{$(b-a)>2\varepsilon$}{
    
    \nl$\Pi^\rho:=\text{left-induced}(\{v_1,\dots,v_n\},\rho)$\;

    \nl\eIf{$|\Pi^\rho|>m$}{
      \nl$a:=\rho$, $\rho:=\frac{a+b}{2}$\;}{
      \nl$\Pi^* := \Pi^\rho$, $b:=\rho$, $\rho:=\frac{a+b}{2}$\;}
  }
    \nl\Return{$\Pi^*$}
    \caption{Optimal left-induced $m$-partition}
    \label{algo:partition} }
  \end{algorithm}
  \DecMargin{.3em}
  \SetAlgorithmName{Trajectory}{ of Trajectories}{Trajectories}
  
  \begin{theorem}[Optimal $m$-partition]\label{Thm:optimal_partition}
    Let $G$ be a chain roadmap. Let $\Pi_m$ be an $m$-partition of
    $G$, and let $\Pi^\rho$ be the left-induced partition of length $\rho$
    of $G$. Then
    \begin{align*}
      \min_{\Pi_m} \operatorname{dim}(\Pi_m) = \min \setdef{\rho \in
        \mathbb{R}_{\ge 0}}{ | \Pi^\rho |\le m }.
    \end{align*}
 \end{theorem}
 \begin{proof}
    Let $\Pi_m$ be an $m$-partition, and let
    $\Pi^\rho=\{\pi_1^\rho,\dots,\pi_k^\rho\}$ be the left induced
    partition of length $\rho$ of a chain roadmap $G$. Let $\rho^* =
    \min_{\Pi_m} \text{dim}(\Pi_m)$. We want to show that $\rho^*$ is
    one of the discontinuity points of the function $|\Pi^\rho|$,
    i.e., that $\rho^*$ verifies the conditions
    \eqref{eq:disc_points}. 

    By contradiction, if $\rho<\rho^*$ and $|\Pi^\rho| \le m$, then an
    $m$-partition with dimension smaller than the optimal would
    exists. Therefore we have $|\Pi^\rho|>m$ if $\rho<\rho^*$.

    Suppose now that $\rho \geq \rho^*$, and let
    $\Pi_m^*=\{\pi_1^*,\dots,\pi_m^*\}$ be an $m$-partition with
    minimum dimension. Notice that $|\pi_1^\rho| \ge |\pi_1^*|$,
    because the cluster $\pi_1^\rho$ contains all the viewpoints
    within distance $\rho$ from $v_1$, and hence also within distance
    $\rho^*$. It follows that $\max \pi_1^\rho \ge \max \pi_1^*$, and
    also that $\min \pi_2^\rho \ge \min \pi_2^*$. By repeating the
    same procedure to the remaining clusters, we obtain that $\max
    \pi_m^\rho \ge \max \pi_m^*$, so that, if $|\Pi^* |=m$ and
    $\rho\ge \rho^*$, then $|\Pi^\rho|\le m$.
  \end{proof}

  Following Theorem \ref{Thm:optimal_partition}, an optimal
  left-induced partition of cardinality (at most) $m$ is also an
  optimal $m$-partition. Notice that for the computation of an optimal
  left-induced partition only the lengths $\rho$ corresponding to the
  discontinuity points of $\Pi^\rho$ need to be considered. Since each
  discontinuity point coincides with the distance between a pair of
  vertices, only $\frac{n(n-1)}{2}$ values need to be
  tested. Therefore, an optimal left-induced partition can be computed
  with complexity $O (n^2)$. In what follows we describe an
  $\varepsilon$-approximation algorithm with linear complexity for any
  $\varepsilon \in \mathbb{R}_{>0}$. Notice that
  $\varepsilon$-approximation algorithms with linear complexity are
  often more convenient for a practical implementation than exact
  algorithms with higher complexity \cite{EMA-RH-AL:06}.

  We now present our algorithm for the computation of an optimal
  $m$-partition. Since the function $\rho \mapsto |\Pi^\rho|$ is
  monotone and continuous, a bisection method is effective for finding
  its discontinuity points, and, therefore, for determining the
  shortest length of a left-induced partition of cardinality $m$. A
  bisection based procedure to compute an optimal left-induced
  partition is in Algorithm \ref{algo:partition}, where the function
  \textit{left-induced($\{v_1,\dots,v_n\}$,$\rho$)} returns the
  left-induced partition defined in equation
  \eqref{eq:left-induced}. We next characterize the convergence
  properties of Algorithm \ref{algo:partition}.

  \begin{lemma}[Convergence of Algorithm
    \ref{algo:partition}]\label{thm:conv_alg_part}
    Let $G$ be a chain roadmap, and let $\Pi_m$ denote an
    $m$-partition of $G$. Let $\rho^*=\min_{\Pi_m}
    \text{dim}(\Pi_m)$. Algorithm \ref{algo:partition} with tolerance
    $\varepsilon$ returns a left-induced partition of dimension at
    most $\rho^* + \varepsilon$ and cardinality at most $m$. Moreover,
    the time complexity of Algorithm \ref{algo:partition} is
    $O(n\log(\varepsilon^{-1}))$.
\end{lemma}
\begin{proof}
  Algorithm \ref{algo:partition} searches for the minimum length
  $\rho^*$ that generates a left-induced partition of cardinality at
  most $m$. Because of Theorem \ref{Thm:optimal_partition}, the length
  $\rho^*$ coincides with one of the discontinuity points of the
  function $|\Pi^\rho|$, and it holds $\rho^* \in (0,2v_n/m)$. Indeed,
  $\rho^*>0$ because $m<n$, and $\rho^*< 2v_n/m$, because
  $(2v_n/m)m>v_n$. Recall from \eqref{eq:disc_points} that
  $|\Pi^\rho|>m$ for every $\rho<\rho^*$, and that the function $\rho
  \mapsto |\Pi^\rho|$ is monotone. Note that the interval $[a,b]$, as
  updated in Algorithm \ref{algo:partition}, contains the value
  $\rho^*$ at every iteration. The length of the interval $[a,b]$ is
  divided by 2 at each iteration, so that, after $\log_2
  (\frac{2v_n}{\varepsilon m})$, the value $\rho^*$ is computed with
  precision $\varepsilon$. Since the computation of $|\Pi^\rho|$ can
  be performed in $O(n)$ operations, the time complexity of Algorithm
  \ref{algo:partition} is $O(n\log(\varepsilon^{-1}))$.
\end{proof}

As a consequence of Proposition \ref{prop:disjoint_sets} and Theorem
\ref{thm:min_refre_time}, in what follows we only consider team
trajectories whose image coincide with an $m$-partition. Therefore,
for ease of notation, we use the set $\{V_1,\dots,V_m\}$ to denote
both the image set of a team trajectory and an $m$-partition of the
chain graph. We conclude this section with a summary of the presented
results.

\begin{theorem}[Patrolling a chain graph at minimum refresh
  time]\label{thm:summary_chain_rt}
  Let $G$ be a chain graph with $n$ viewpoints and let $m \le n$ be
  the number of robots. Let $\mathcal{V}$ be an optimal $m$-partition
  of $G$ computed by means of Algorithm \ref{algo:partition} with
  tolerance $\varepsilon$. Let $\subscr{d}{max}$ be the dimension of
  $\mathcal{V}$. A team trajectory with image $\mathcal{V}$, and
  minimum refresh time $2 \subscr{d}{max}$ is computed as in
  Trajectory \ref{algo:max_sweep}. Moreover, the time complexity of
  designing such trajectory is $O(n\log(\varepsilon^{-1}))$.
\end{theorem}


\section{Minimum refresh time and latency team trajectory on a chain
  roadmap}\label{sec:latency}
The previous section considers the problem of designing team
trajectories with optimal refresh time on a chain graph. In a
patrolling mission it may be important for the robots to communicate
with each other in order to gather information about the status of the
entire environment. For instance, messages could be sent by a unit to
ask for reinforcement, or to spread an alarm. Informally, we call
\emph{latency} of a team trajectory $X$, in short $\text{LT}(X)$, the
shortest time interval necessary for a message generated by any robot
to reach all the other robots. In other words, given our communication
model, the latency of a team trajectory is a measure of how fast a
message spreads to all robots. In this section we describe team
trajectories with minimum refresh time and latency.

We now give a more formal definition of $\text{LT}(X)$. Recall that,
by assumption, two robots are allowed to communicate when they lie on
two adjacent viewpoints. In a chain roadmap, for a message to reach
every robot in the chain, every pair of adjacent robots needs to
communicate. For $i\in\{2,\dots,m\}$, let $\Phi_i$ denote the union of
the set of times at which the robots $i-1$ and $i$ communicate and
$\{0\}$.
The \emph{up-latency} of $X$, in short $\textup{LT}_{\textup{up}}
(X)$, is the longest time interval between any two consecutive
communications between the robots $1,2$ and
$m-1,m$. Precisely,
\begin{align*}
  \textup{LT}_{\textup{up}} (X) = \max_{t_2 \in \Phi_2} \min_{t_m \in
    \bar\Phi_{m}(t_2)} t_m - t_2,
\end{align*}
where $\bar\Phi_m(t_2) = \setdef{ t_m \in \Phi_m }{ \exists \, t_3 \in
  \Phi_3, \dots, t_{m-1} \in \Phi_{m-1}, t_2 \le t_3 \le \dots \le
  t_m} \cup \{ \subscr{T}{f} \}$. Analogously, we call
\emph{down-latency} the quantity
\begin{align*}
  \textup{LT}_{\textup{down}} (X) = \max_{t_m \in \Phi_m} \min_{t_2
    \in \bar\Phi_{2}(t_m)} t_2 - t_m,
\end{align*}
where $\bar\Phi_2(t_m) = \setdef{t_2 \in \Phi_2}{\exists \, t_3 \in
  \Phi_3, \dots, t_{m-1} \in \Phi_{m-1}, t_2 \ge t_3 \ge \dots \ge
  t_m} \cup \{ \subscr{T}{f} \}$. Finally, we define the latency of a
team trajectory as
\begin{align*}
  \textup{LT}(X) = \max \{ \textup{LT}_{\textup{up}} (X), \textup{LT}_{\textup{down}} (X)\}.
\end{align*} 
Notice that our definitions of latency hold for $m \ge 2$, and that,
if $m = 2$, then we have $\textup{LT}_{\textup{up}} (X) =
\textup{LT}_{\textup{down}} (X) = \textup{LT}(X) = 0$ for every team
trajectory $X$.  We envision that the up- and down-latency performance
criteria should be adopted when it is of interest to report the status
of the monitored area to a base station located at one end of the
chain environment. The latency minimization problem is more
appropriate for fully distributed scenarios. In this section we design
synchronized team trajectories with the following two features. First,
since a minimum refresh time trajectory is determined by an optimal
partition of the chain graph, we aim at finding team trajectories
associated with the same optimal partition.\footnote{We focus on this
  family of trajectories, and we leave the more general optimization
  problem as the subject of future research. However, this family of
  trajectories is believed to contain an (unconstrained) optimal
  solution.} Second, we design synchronized team trajectories with
minimum up-latency (resp. down-latency) or latency.

\subsection{Lower bound and optimal team trajectory for up-latency}
We start by showing a lower bound for $\textup{LT}_{\textup{up}} (X)$ and
$\textup{LT}_{\textup{down}} (X)$. Recall that, for a partition
$\{V_1,\dots,V_m\}$, we have $l_i = \min_{v \in V_i} v$, $r_i = \max_{v \in
  V_i} v$, $d_i = r_i - l_i$, and $d_\textup{max} = \max_{i\in
  \{1,\dots,m\}} d_i$.

\begin{lemma}[Up-latency lower bound]\label{lower_bound}
  Let $G$ be a chain roadmap, and let $\{V_1,\dots,V_m\}$ be an
  $m$-partition of $G$. The latency of a team trajectory with image
  $\{V_1,\dots,V_m\}$ is lower bounded by $\sum_{i=2}^{m-1} d_i$.
\end{lemma}
\begin{proof}
  The proposition follows from the fact that the robots speed is
  bounded by $1$, and that the robots need to travel their segment to
  communicate with the neighboring robots.
\end{proof}

For the up-latency of a team trajectory to equal the lower bound in
Lemma \ref{lower_bound}, each robot $i$ needs to transfer a message
from robot $i-1$ to robot $i+1$ in time $d_i$. In order to do so, each
robot $i$ needs to communicate with its neighbor $i+1$ as soon as $x_i
(t) = r_i$.


\begin{theorem}[Patrolling a chain graph at minimum refresh time and
  minimum up-latency]
  Let $G$ be a chain graph with $n$ viewpoints and let $m \le n$ be
  the number of robots. Let $\mathcal{V}$ be an optimal $m$-partition
  of $G$ computed by means of Algorithm \ref{algo:partition} with
  tolerance $\varepsilon$. Let $\subscr{d}{max}$ be the dimension of
  $\mathcal{V}$, and let $d_i$ be the length of the $i$-th cluster. A
  team trajectory with image $\mathcal{V}$, minimum refresh time $2
  \subscr{d}{max}$, and minimum up-latency $\sum_{j=2}^{m-1} d_j$ is
  computed as in Trajectory \ref{algo:min_base}. Moreover, the time
  complexity of designing such trajectory is
  $O(n\log(\varepsilon^{-1}))$.
\end{theorem}
\begin{proof}
  The theorem follows by observing that the trajectory is
  $2{\subscr{d}{max}}$-periodic, and that no robot $i$ waits at $r_i$
  to communicate with the neighboring robot $i+1$. The up-latency
  equals the lower bound in Lemma \ref{lower_bound}, and it is
  therefore minimum. Regarding the computational complexity, notice
  that it is determined by the computation of the optimal
  $m$-partition $\mathcal{V}$, and hence, by Lemma
  \ref{thm:conv_alg_part}, it equals $O(n\log(\varepsilon^{-1}))$.
\end{proof}

An example of a team trajectory with minimum refresh time and minimum
up-latency is in Fig. \ref{fig:traj3}. Finally, observe that the
minimization of the down-latency can be achieved in an analogous way.

\IncMargin{.3em}
\begin{algorithm}[t]
  {\footnotesize
  \SetKwData{Left}{left}
  \SetKwData{Right}{right}
  \SetKwData{Position}{position}
  \SetKwInOut{Input}{Input}
  \SetKwInOut{Output}{Output}
  \SetKwInOut{Set}{Set}
  \SetKwInOut{Require}{Require}
  \SetKwInOut{Title}{Algorithm}
  \Input{$l_i := \min_{v \in V_i} v$, $r_i := \max_{v \in V_i} v$,
    $d_i := r_i - l_i$, $d_\textup{max} := \max_j r_j -l_j$\;}
  \Require{optimal partition of the chain graph\;} \Set{$t_0 :=
    -\sum_{j=1}^{m-1} d_j$, $k \in \mathbb{N} \mapsto T_i(k) :=
    2kd_\textup{max} + t_0 + \sum_{j = 1}^{i-1} d_j$\;}

\BlankLine
    
    \nl$x_i (t) := l_i$ for $t_0 \le t \le T_i(0)$ and for $T_i(k) + 2d_i \le t \le T_i(k+1)$\;

    \nl$x_i (t) := r_i$ for $t := T_i(k) + d_i$\;

      \caption{\textit{Minimum base-latency team trajectory ($i$-th robot)}}
      \label{algo:min_base}}
\end{algorithm} 
\DecMargin{.3em}

\begin{figure}
	\centering
	\includegraphics[width=.5\columnwidth]{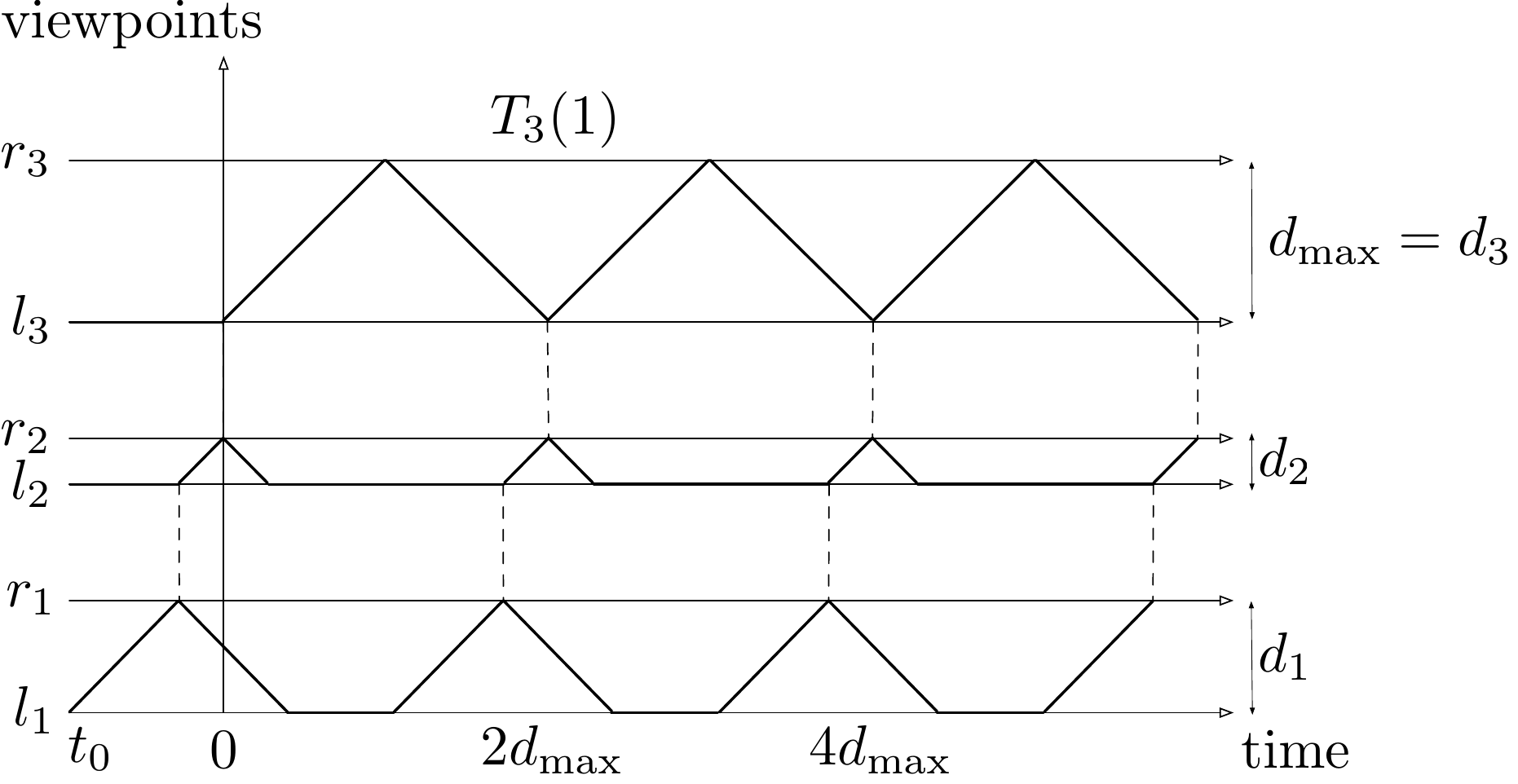}
	\caption{A team trajectory with minimum up-latency generated
          with the procedure in Trajectory \ref{algo:min_base}. Notice
          that each robot $i$ communicates with the neighboring robot
          $i+1$ as soon as $x_i (t) = r_i$. A team trajectory with
          minimum down-latency is obtained by shifting in time the
          trajectory of robot $1$, in a way that robot $2$
          communicates with robot $1$ as soon as $x_2 (t) =
          l_2$.}\label{fig:traj3}
\end{figure}

\begin{remark}[Simplifications for clusters of equal length]
  Let $\{V_1,\dots,V_m\}$ be an optimal $m$-partition, and suppose
  that $d_\textup{max}=r_i-l_i$ for all $i\in \{1,\dots,m\}$. Then the
  up-latency and the down-latency can be made minimum and equal to
  $(m-2) d_\textup{max}$ by arranging the robots trajectories to be in
  opposite phase. Specifically, for $k = 0,2,\dots$, we set
  \begin{align*}
    x_i (2 k d_\textup{max}) &= l_i,\;\;
    x_i ((2k+1)d_\textup{max}) = r_i,
  \end{align*}
  if $i$ is odd, and
  \begin{align*}
    x_i (2 k d_\textup{max}) &= r_i,\;\;
    x_i ((2k+1)d_\textup{max}) = l_i,
  \end{align*}
  if $i$ is even. Because of Lemma \ref{lower_bound}, the above
  trajectory has minimum latency. This particular case was studied in
  \cite{DBK-RWB-RSH:08}.
\end{remark}

\subsection{Lower bound for latency}
We now consider the minimization of the latency criterion, and we
restrict our attention to periodic team trajectories. To be more
precise, let $\{V_1,\dots,V_m\}$ be an optimal $m$-partition of the
environment, and let $d_\textup{max}$ denote the longest length of the
clusters. We aim at finding a $2d_\textup{max}$-periodic team trajectory
with image $\{V_1,\dots,V_m\}$ and minimum latency. Notice that, by
imposing a periodicity of $2d_\textup{max}$, the refresh time of the
trajectory, if finite, is also minimized.

We start by considering the pedagogical situation in which $d_i +
d_{i+1} > d_\textup{max}$ for all $i \in \{1,\dots,m-1\}$. In the next
Lemma, we show that the frequency of message exchange among the robots
is limited by the periodicity of the trajectory.

\begin{lemma}[Frequency of message
  exchange]\label{lemma:exchange_rate}
  Consider a $2d_\textup{max}$-periodic team trajectory, and let $d_i
  + d_{i+1} > d_\textup{max}$ for all $i \in \{1,\dots,m-1\}$. For any
  $t \in [0,\subscr{T}{f} - 2d_\textup{max}]$ and for any $i \in
  \{2,\dots,m-2\}$, there exist no two distinct sequences
  $t_i,t_{i+1},t_{i+2}$ and $\bar t_i,\bar t_{i+1},\bar t_{i+2}$, with
  $t_j, \bar t_j \in \Phi_j$, $j = i, i+1, i+2$,
  such that
  \begin{align*}
    &t \le t_i \le t_{i+1} \le t_{i+2} \le t + 2d_\textup{max}\\
    &t \le \bar t_i \le \bar t_{i+1} \le \bar t_{i+2} \le t +
    2d_\textup{max}\\
    &t_{i+1} \le \bar t_i\\
    &t_{i+2} \le \bar t_{i+1}.
  \end{align*}
  Moreover, for any $i \in \{m,\dots,4 \}$, there exist no two
  sequences $t_i,t_{i-1},t_{i-2}$ and $\bar t_i,\bar t_{i-1},\bar
  t_{i-2}$ with $t_j, \bar t_j \in \Phi_j$, $j = i,i-1,i-2$,
  such that
  \begin{align*}
    &t \le t_i \le t_{i-1} \le
    t_{i-2} \le t + 2d_\textup{max}\\
    &t \le \bar t_i \le \bar t_{i-1} \le \bar t_{i-2} \le t +
    2d_\textup{max}\\
    &t_{i-1} \le \bar t_i\\
    &t_{i-2} \le \bar t_{i-1}.
  \end{align*}
\end{lemma}
\begin{proof}
  Since $d_{i+1} + d_{i+2} > d_\textup{max}$, it follows
  $\max\{d_{i+1},d_{i+2}\} > d_\textup{max}/2$. Let $d_{i+1} >
  d_\textup{max}/2$. By contradiction, if two distinct sequences
  $t_i,t_{i+1},t_{i+2}$ and $\bar t_i,\bar t_{i+1},\bar t_{i+2}$
  exist, with $t \le t_i \le t_{i+1} \le t_{i+2} \le t +
  2d_\textup{max}$, $t \le \bar t_i \le \bar t_{i+1} \le \bar t_{i+2}
  \le t + 2d_\textup{max}$, $t_{i+1} \le \bar t_i$ and $t_{i+2} \le
  \bar t_{i+1}$, then the $(i+1)$-th robot travels its cluster four
  times. Since the speed of the robots is bounded by one, robot $i+1$
  cannot travel its cluster four times in a period of
  $2d_\textup{max}$ (cf. Fig. \ref{fig:frequency_exchange}). The
  second part of the theorem follows from an analogous reasoning.
\end{proof}

\begin{figure}
    \centering
    \includegraphics[width=.5\columnwidth]{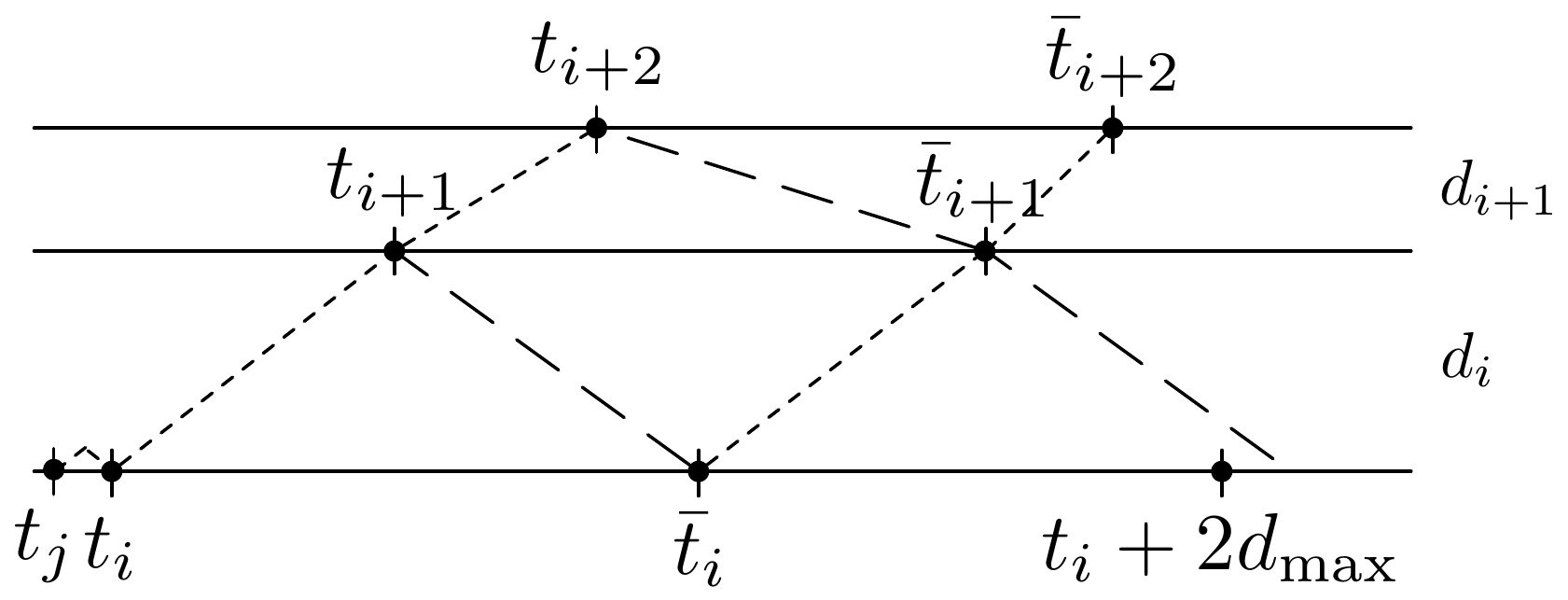}
    \caption{As stated in Lemma \ref{lemma:exchange_rate}, at most one
      communication sequence is possible within each period of length
      $2d_{\text{max}}$. Here $d_i + d_{i+1} > d_{\text{max}}$.}
    \label{fig:frequency_exchange}
\end{figure}


Notice that in the above Lemma the index $i$ belongs to the set
$\{2,\dots,m-2\}$ (resp. $\{m,\dots,4\}$) because we consider $3$
consecutive communication instants, and because $\Phi_i$ denotes the
sequence of times at which robots $i-1$ and $i$ communicate. Because
of Lemma \ref{lemma:exchange_rate}, in a $2d_\textup{max}$-periodic
team trajectory with $d_i + d_{i+1} > d_\textup{max}$ only one message
can be passed from robot $i$ to robot $i+3$ in a period of time of
$2d_\textup{max}$. This limitation determines a lower bound on the
latency of a periodic trajectory. Notice that eventual communication
instants $t_j \in \Phi_i$, with $t \le t_j \le t_i \le t_{i+1}$, do
not influence the latency, since all information can be transmitted at
time $t_i$ without affecting the latency performance.

\begin{lemma}[Latency lower bound, simple
  case]\label{lemma:lower_bound_latency}
  Let $X$ be a $2d_\textup{max}$-periodic team trajectory with $d_i +
  d_{i+1} > d_\textup{max}$ for all $i \in \{1,\dots,m-1\}$. Then
  \begin{align*}
    \min_X \operatorname{LT}(X) \ge (m-2) d_\textup{max}.
  \end{align*}
\end{lemma}
\begin{proof}
  Because of Lemma \ref{lemma:exchange_rate}, a message can be
  transfered from robot $i$ to robot $i+3$ in at most $2d_\textup{max}$
  instants of time, by traveling the clusters $V_{i+1}$ and
  $V_{i+2}$. Without losing generality, we let $d_\textup{max}$ be the
  time to pass a message form $i+1$ to $i+2$, and from $i+2$ to
  $i+3$. Notice that the same reasoning holds also for the time to
  transfer a message from $i+3$ to $i$. Therefore, the latency is
  lower bounded by $(m-2) d_\textup{max}$.
\end{proof}

We now consider the situation in which an optimal $m$-partition does
not verify the constraint $d_i + d_{i+1} > d_\textup{max}$ for all $i
\in \{1,\dots,m-1\}$. Intuitively, for what concerns the latency, two
consecutive clusters with length $d_i$ and $d_{i+1}$ may be regarded
as one single cluster of length $d_i + d_{i+1}$. Therefore, in order
to use Lemma \ref{lemma:exchange_rate} and Lemma
\ref{lemma:lower_bound_latency}, we partition the clusters
$\{V_1,\dots,V_m\}$ into groups such that the sum of the length of the
clusters of two consecutive groups is greater than $d_\textup{max}$. In
other words, let $\bar V_r =\{\bar r_1,\dots,\bar r_{\bar m}\}$ be the
set of \emph{right-extreme} viewpoints of the partition
$\{V_1,\dots,V_m\}$ defined recursively as
\begin{align*}
  \bar r_0 &:= l_1,\\
  \bar r_i &:= \max_{j\in\{1,\dots,m\}} \setdef{ r_j }{ \sum_{k =
      p_i}^j d_k \le \subscr{d}{max}},\;\; i=1,\dots,\bar m,
\end{align*}
where $p_i = \min \{1,\dots,m\}$ such that $l_{p_i} \ge \bar
r_{i-1}$. Let $\bar V_l=\{\bar l_1,\dots,\bar l_{\bar m}\}$ be the set
of \emph{left-extreme} viewpoints defined recursively as
\begin{align*}
  \bar l_1 &:= l_1,\\
  \bar l_i &:= \min_{j\in \{1,\dots,m\}} \setdef{ l_j }{ l_j \ge \bar
    r_{i-1}}, \;\; i=2,\dots,\bar m.
\end{align*}
Additionally, define the set of aggregated clusters $\bar V = \{ \bar
V_1,\dots, \bar V_{\bar m} \}$, where $\bar V_i$ contains all the
clusters within $\bar l_i$ and $\bar r_i$, and let $\bar d_i$ be the
sum of the length of the clusters in $\Bar V_i$.

\begin{lemma}[Latency lower bound, general
  case]\label{thm:lat_general}
  Let $X$ be a $2d_\textup{max}$-periodic team trajectory with image
  $\{V_1,\dots,V_m\}$, and let $\bar m$ be the number of aggregated
  clusters. Then,
  \begin{align*}
    \min_X \operatorname{LT}(X) \ge (\bar m -2) d_\textup{max} + (\bar d_1 -
    d_1) + (\bar d_{\bar m} - d_m).
  \end{align*}
\end{lemma}
\begin{proof}
  Consider the clusters defined by the right extreme viewpoints, and
  notice that they verify $d_i + d_{i+1} > d_\textup{max}$. Then, the
  Theorem follows from Lemma \ref{lemma:lower_bound_latency}, and from
  the fact that the minimum latency on the image $\{\bar
  V_1,\dots,\bar V_{\bar m}\}$ equals the minimum latency on the image
  $\{V_1,\dots,V_m\}$. The terms $\bar d_1 - d_1$ and $\bar d_{\bar m}
  - d_m$ are due to the fact that we are interested in delivering a
  message from robot $1$ to robot $m$ in the original configuration,
  and not on the aggregated chain.
\end{proof}

\subsection{Optimal team trajectory for latency}\label{sec:realistic_case}
A team Trajectory with minimum refresh time and minimum latency is
formally presented in Trajectory \ref{algo:opt_traj}, where we specify
the instants of time at which a robot changes its velocity, and we
assume that it moves at maximum speed otherwise. An example is
reported in Fig. \ref{example_latency_opt}, and here we give an
informal description. Let $\{V_1,\dots,V_m\}$ be an optimal
$m$-partition of a chain graph, and let $\{\bar V_1, \dots, \bar
V_{\bar m}\}$ be the set of aggregated clusters. Recall that $\bar
V_i$ is a subset of $\{V_1,\dots,V_m\}$, and that the sum of the
length of two consecutive elements is larger than
$d_\textup{max}$. The procedure in Trajectory \ref{algo:opt_traj}
(lines $6-11$ and $18-23$) is such that the robots in the same group
behave as a single robot assigned to the whole set of viewpoints. In
other words, the motion of the robots in the same group is determined
by a token passing mechanism, in which robot $i+1$ moves towards
$r_{i+1}$ only when $x_i (t) = r_i$ and $x_i (t^-) \neq r_i$, and,
analogously, robot $i$ moves towards $l_{i}$ only when $x_i (t) = l_i$
and $x_i (t^-) \neq l_i$. Instead, lines $1-5$ and $12-17$ in
Trajectory \ref{algo:opt_traj} guarantee the transfer of one message
in a period of $2d_\textup{max}$ between three consecutive groups,
and, consequently, the minimization of the latency. Indeed, since the
sum of the length of two consecutive groups is larger than
$d_\textup{max}$, because of Lemma \ref{lemma:exchange_rate}, no more
than one message can be transferred between three consecutive groups
in a period of $2d_\textup{max}$.

\begin{theorem}[Patrolling a chain graph at minimum refresh time and
  minimum latency]
  Let $G$ be a chain graph with $n$ viewpoints and let $m \le n$ be
  the number of robots. Let $\mathcal{V}$ be an optimal $m$-partition
  of $G$ computed by means of Algorithm \ref{algo:partition} with
  tolerance $\varepsilon$. Let $\subscr{d}{max}$ be the dimension of
  $\mathcal{V}$, and let $d_1$ (resp. $d_m$) be the length of the
  first (resp. last) cluster in $\mathcal{V}$. Let $\bar m$ be the
  number of aggregated clusters, and let $\bar d_1$ (resp. $\bar
  d_{\bar m}$) be the length of the first (resp. last) aggregated
  cluster. A team trajectory with image $\mathcal{V}$, minimum refresh
  time $2 \subscr{d}{max}$, and minimum latency $(\bar m -2)
  d_\textup{max} + \bar d_1 - d_1 + \bar d_{\bar m} - d_m$ is computed
  as in Trajectory \ref{algo:opt_traj}. Moreover, the time complexity
  of designing such trajectory is $O(n\log(\varepsilon^{-1}))$.
\end{theorem}
\begin{proof}
  By inspection, the team Trajectory described in Trajectory
  \ref{algo:opt_traj} is $2d_\textup{max}$-periodic, and therefore it
  has minimum refresh time. Moreover, by construction, the
  communications at the extreme viewpoints happen every
  $2d_\textup{max}$ instants of time, so that the latency is equal to
  $(\bar m -2) d_\textup{max} + \bar d_1 - d_1 + \bar d_{\bar m} -
  d_m$, and hence, by Lemma \ref{thm:lat_general}, minimum. Regarding
  the computational complexity, notice that it is determined by the
  computation of the optimal $m$-partition $\mathcal{V}$, and hence,
  by Lemma \ref{thm:conv_alg_part}, it equals
  $O(n\log(\varepsilon^{-1}))$.
\end{proof}

\IncMargin{.3em}
\begin{algorithm}[t]
  {\footnotesize
  \SetKwData{Left}{left}
  \SetKwData{Right}{right} 
  \SetKwData{Position}{position} 
  \SetKwInOut{Input}{Input}
  \SetKwInOut{Set}{Set}
  \SetKwInOut{Require}{Require}
  \Input{$l_i := \min_{v \in V_i} v$, $r_i := \max_{v \in V_i} v$, 
    $d_\textup{max} := \max_j r_j - l_j$, $R :=$ set containing the identifiers
    of the robots in the same group as $i$, $k :=$ index of the
    aggregated cluster $\{1,\dots,\bar m\}$\;}
  \Require{optimal partition of the chain graph\;}
    \BlankLine

    \nl\Case{$l_i$ is a left-extreme}{
      \nl\If{$k$ is odd}{
      
        \nl$x_i(t)=l_i$ with $t=0,2d_\textup{max},4d_\textup{max},\dots$\;
        
      }
      \nl\ElseIf{$k$ is even}{

        \nl$x_i(t)=l_i$ with $t=d_\textup{max},3d_\textup{max},5d_\textup{max},\dots$\;

      }
    }
    \nl\Case{$l_i$ is not a left-extreme}{
      \nl$\delta_i:= \sum_{j \in R, j < i} d_j $\;
      \nl\If{$k$ is odd}{
      
        \nl$x_i(t)=l_i$ for $\tau -\delta_i \le t \le \tau + \delta_i$, with
        $\tau = 0,2d_\textup{max},4d_\textup{max},\dots$\;
 
      }
      \nl\ElseIf{$k$ is even}{

        \nl$x_i(t)=l_i$ for $\tau -\delta_i \le t \le \tau + \delta_i$, with
        $\tau = d_\textup{max},3d_\textup{max},5d_\textup{max},\dots$\;

      }
    }

   \nl\Case{$r_i$ is a right-extreme}{
     \nl$\delta_i :=  d_\textup{max} - \sum_{j \in R} d_j\;$
     
     \nl\If{$k$ is odd}{
      
        \nl$x_i(t)=r_i$ for $\tau - \delta_i \le t \le \tau + \delta_i$,
        with $\tau = d_\textup{max},3d_\textup{max},5d_\textup{max},\dots$\;

      }
      \nl\ElseIf{$k$ is even}{

        \nl$x_i(t)=r_i$ for $\tau - \delta_i \le t \le \tau + \delta_i$,
        with $\tau = 0,2d_\textup{max},4d_\textup{max},\dots$\;

      }
    }

    \nl\Case{$r_i$ is not a right-extreme}{
      \nl$\delta_i:= d_\textup{max} - \sum_{j \in R, j \le i} d_j $\;
      \nl\If{$k$ is odd}{
      
        \nl$x_i(t)=r_i$ for $\tau - \delta_i \le t \le \tau + \delta_i$,
        with $\tau = d_\textup{max},3d_\textup{max},5d_\textup{max},\dots$\;

      }
      \nl\ElseIf{$k$ is even}{

        \nl$x_i(t)=r_i$ for $\tau - \delta_i \le t \le \tau + \delta_i$,
        with $\tau = 0,2d_\textup{max},4d_\textup{max},\dots$\;
      }
    }
    \caption{\textit{Minimum refresh time and latency trajectory
        ($i$-th robot)}}
 \label{algo:opt_traj}}
\end{algorithm}
\DecMargin{.3em}

\begin{figure}
	\centering
	\includegraphics[width=.55\columnwidth]{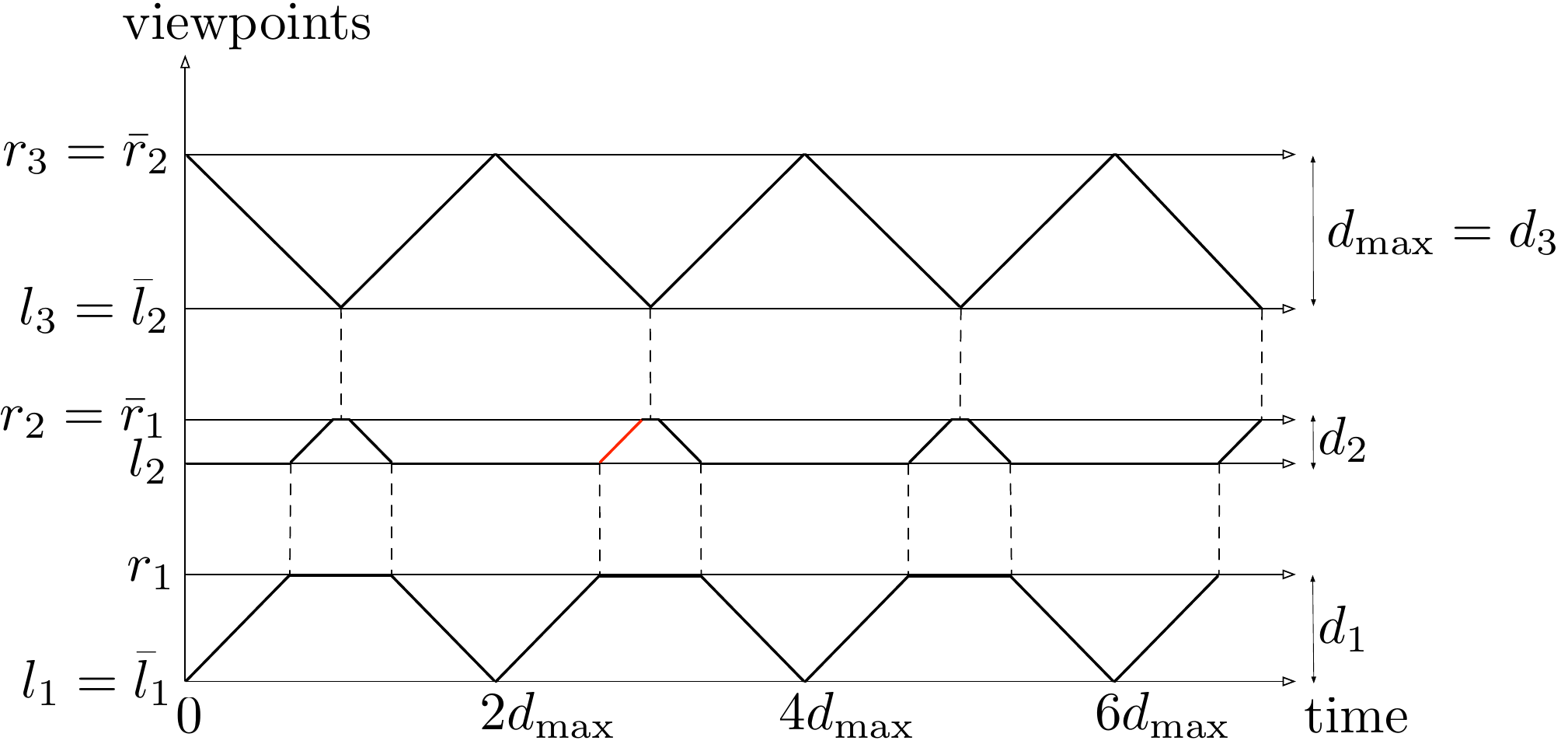}
	\caption{A team trajectory with minimum refresh time
          $(2d_\textup{max})$ and minimum latency $((\bar m -
          2)d_\textup{max} + d_1 + d_2 - d_1 + d_3 - d_3 = d_2)$ is
          here reported. A message with delivery time equal to $d_2$
          is reported in red. The right-extreme viewpoints are $\{\bar
          r_1,\bar r_2\}$, the left-extreme viewpoints are $\{\bar
          l_1,\bar l_2\}$, and $\bar m = 2$. The represented team
          trajectory has minimum refresh time $(2d_\textup{max})$ and
          minimum latency $((\bar m - 2)d_\textup{max} + (d_1 + d_2) -
          d_1 + d_3 - d_3 = d_2)$. A message with delivery time equal
          to $d_2$ is reported in red.}\label{example_latency_opt}
\end{figure}

\section{Distributed synchronization algorithm on a chain
  roadmap}\label{sec:distr_alg}
In the previous sections we have shown that, for the computation of a
minimum refresh time and latency team trajectory, first an optimal
$m$-partition of the roadmap needs to be found, and, then, a
synchronization of the motion of the robots needs to be achieved to
ensure communication between neighboring robots. The distributed
computation of an optimal $m$-partition follows directly from
Algorithm \ref{algo:partition}, by letting the robots compute the
left-induced partition of length $l$ in a distributed way. A simple
solution consists of the following three operation:
\begin{enumerate}
  \item the robots gather at the leftmost viewpoint, and
  \item determine the cardinality of the team and elect a
    leader;
  \item the leader computes an optimal left-induced partition, and
    assigns a different cluster to each robot.
\end{enumerate}
Notice that, by only assuming the capability of detecting the
viewpoints in the roadmap (in particular the leftmost and rightmost
viewpoint) the robots can distributively compute an optimal
partition. Indeed, the leader can simply travel the roadmap, and
evaluate if, for a given length $\rho$, the cardinality of the
corresponding left-induced partition is greater, equal, or smaller
than the cardinality of the team. We believe that, by simple
programming operations, the above procedure can be improved, so as to
handle reconfigurations due to addition or removal of robots in the
team. In this section we focus instead on the design of a distributed
feedback law to synchronize the motion of the robots so as to
guarantee the minimization of the latency of the trajectory.

Recall that $x_i(t)$ denotes the position on the chain of the robot
$i$ at time $t$. Moreover, let the $i$-th cluster of an optimal
$m$-partition be delimited by $l_i$ and $r_i$. Let $\text{dir}_i \in
\{-1,0,1\}$ denote the direction of motion of robot $i$. Let
\textit{c-time} denote the current time of simulation, let
\textit{a-time} be the time at which a robot arrives at his right
extreme. Let \textit{n-meet(i,j)} be a function that returns the
number of times that the robots $i$ and $j$ have communicated. Let
\textit{Timer($\delta$)} be a function that returns $1$ after a time
interval of length $\delta$. An algorithm for the robots to
distributively converge to a minimum refresh time and latency team
trajectory is in Algorithm \ref{algo:synchro}. It should be noticed
that the algorithm assumes the knowledge of an optimal partitioning of
the chain graph, and of the left- and right-extreme sets.

Algorithm \ref{algo:synchro} is informally described as
follows. First, the velocity of a robot changes only when its position
coincides with the leftmost or the rightmost of the assigned
viewpoints. When a robot reaches an extreme viewpoint, it waits until
a communication with the neighboring robot happens (lines $7-8$). This
mechanism determines the feedback behavior of our procedure. The
behavior of a robot after a communication is determined by the lines
$9-21$, which reproduce the optimal behavior described in Trajectory
\ref{algo:opt_traj}. To be more precise, the function
\textit{Token($i,j$)} coordinates the motion of the robots in the same
group (see Section \ref{sec:realistic_case}), so that they move as if
there was a single robot sweeping the viewpoints in the same
group. The function \textit{Timer($\delta_i$)}, instead, ensures the
maximum frequency of messages exchange between consecutive groups, so
as to minimize the latency of the resulting trajectory.

\SetAlgorithmName{Algorithm}{ of Algorithms}{Algorithms}
\IncMargin{.3em}
\begin{algorithm}[t]
  {\footnotesize
    \SetKwInOut{Input}{Input}
    \SetKwInOut{Output}{output}
    \SetKwInOut{Title}{Algorithm}
    \SetKwInOut{Set}{Set}
    \SetKwInOut{Require}{Require}
    \Input{$l_i := \min_{v \in V_i} v$, $r_i := \max_{v \in V_i} v$,
      $d_\text{max} := \max_j r_j - l_j$, $R :=$ set containing the identifiers
      of the robots in the same group as $i$\;}  
    \Set{$\delta_i: = (d_\text{max} - \sum_{j\in R} d_j)/2$,
      $\text{dir}_i \in \{1,-1\}$\;}

    \Require{optimal partition of the chain graph; $l_i \le x_i(0) \le r_i$\;}
    \BlankLine
 
  \nl\While{true}{
    \nl\Case{$i = 1$ and $x_1(t) = l_1$}{
      \nl$\text{dir}_1:=1$\;
      }
      \nl\Case{$i = m$ and $x_m(t) = r_m$}{
      \nl$\text{dir}_m:=-1$\;
      }
    \nl\Case{($x_i(t) = l_i$ and $x_{i-1}(t) \neq r_{i-1}$) or
      \mbox{$\;\;\;\;\;\;\;$($x_i(t) = r_i$ and $x_{i+1}(t) \neq l_{i+1}$)}}{ 
      
      \nl$\text{dir}_i:=0$\;
    }
      


          

          

    \nl\Case{$x_i(t) = l_i$ and $x_{i-1}(t) = r_{i-1}$}{
      \nl\If{$l_i\in \bar V_l$}{
        
        \nl receive $\tau$ from robot $i-1$\;
        
        \nl$\text{dir}_i:=\text{Timer}(\tau)$\;
      }
      \nl\ElseIf{$l_i \not\in \bar V_l$}{\nl$\text{dir}_i:=\text{Token}(i,i-1)$\;}
    }
    \nl\Case{$x_i(t) = r_i$ and $x_{i+1}(t) = l_{i+1}$}{
     \nl $\tau:=\max\{0, \textit{a-time}+\delta_i-\textit{c-time}
     \}$\;
     
     \nl send $\tau$ to robot $i+1$\;
     
     \nl\If{$r_i\in \bar V_r$}{
       
       \nl$\text{dir}_i:=-\text{Timer}(\delta_i + \tau)$\;}
     \nl\ElseIf{$r_i \not\in \bar V_r$}{\nl$\text{dir}_i:=\text{Token}(i, i+1)$\;}
   } 
 \nl$\dot x_i (t) :=\text{dir}_i$\;
  }
  \caption{\textit{Minimum refresh
      time and latency team trajectory ($i$-th robot)}}
 \label{algo:synchro}}
\end{algorithm} 
\DecMargin{.3em}
\SetAlgorithmName{Trajectory}{ of Trajectories}{Trajectories}

\IncMargin{.3em}
\begin{function}[t]
  {\footnotesize
  \SetKwData{Left}{left} 
  \SetKwData{Right}{right} 
  \SetKwData{Position}{position} 
  \SetKwInOut{Input}{Input}
  \SetKwInOut{Output}{output}
  \SetKwInOut{Title}{Function}
  \SetKwInOut{Set}{Set}
  \BlankLine
  \nl\lCase{$i>j$}{
    \lIf{\text{n-meet(i,j)} is even}{\Return{1}}
  }
  
 \nl\lCase{$i<j$}{
    \lIf{\text{n-meet(i,j)} is odd}{\Return{-1}}
  }
  
  \nl\lOther{\Return{0}}
 \caption{Token(i,j)}
 \label{algo:token}}
\end{function} 
\DecMargin{.3em}

\begin{theorem}[Optimal team trajectory]\label{thm:latency}
  Let $X$ be the team trajectory generated by Algorithm
  \ref{algo:synchro}. There exists a finite time after which $X$ has
  minimum refresh time and latency.
\end{theorem}
\begin{proof}
  Let $\{V_1,\dots,V_m\}$ be an optimal $m$-partition, let $\bar V_r =
  \{\bar r_1 , \dots, \bar r_{C^*}\}$ be the set of right-extreme
  viewpoints, and let $\bar V_l =\{\bar l_1 , \dots, \bar l_{C^*}\}$
  be the set of left-extreme viewpoints. Let $R$ denote the set of
  robots patrolling a viewpoint between $\bar l_i$ and $\bar r_i$,
  where $\bar l_i$ and $\bar r_i$ are as previously defined. First,
  notice that robot $1$ (resp. $m$) sets $\dot x_1 (t) = 1$
  (resp. $\dot x_m (t) = -1$) as soon as $x_1 (t) = l_1$ (resp. $x_m
  (t) = r_m$). Second, the function $\textit{Token}(i,j)$ guarantees
  that, when $i,j \in R$ communicate, exactly one robot among $i , j$
  maintains a zero velocity and in an alternate way. Therefore, after
  a finite time $T_i$, independent of the initial robots motion
  direction, the velocities of the robots in $R$ are such that, upon
  communication, $\dot x_i (t) = \dot x_{i-1} (t^{-})$ and $\dot
  x_{i-1} (t) = 0$. In other words, after $T_i$, the robots in $R$
  behave as a single robot sweeping the segments between $\bar l_i$
  and $\bar r_i$. Finally, the function $\textit{Timer}(\tau)$ and the
  parameter $\tau$ guarantee that the communications at the extreme
  viewpoints happen every $2d_\text{max}$ instants of time. We
  conclude that the trajectory generated by Algorithm
  \ref{algo:synchro} converges in finite time to a team trajectory
  with minimum refresh time and latency.
\end{proof}

It should be observed that, differently from the team trajectories
presented in the previous sections, Algorithm \ref{algo:synchro}
contains a feedback procedure to synchronize the motion of the
robots. Our algorithm is independent of the initial configuration,
and, as it is shown in the next section, it is robust to a certain
class of robot failures.

\section{A case study}\label{sec:simulations}
\begin{figure}[tb!]
  \centering
  \includegraphics[width=0.5\columnwidth]{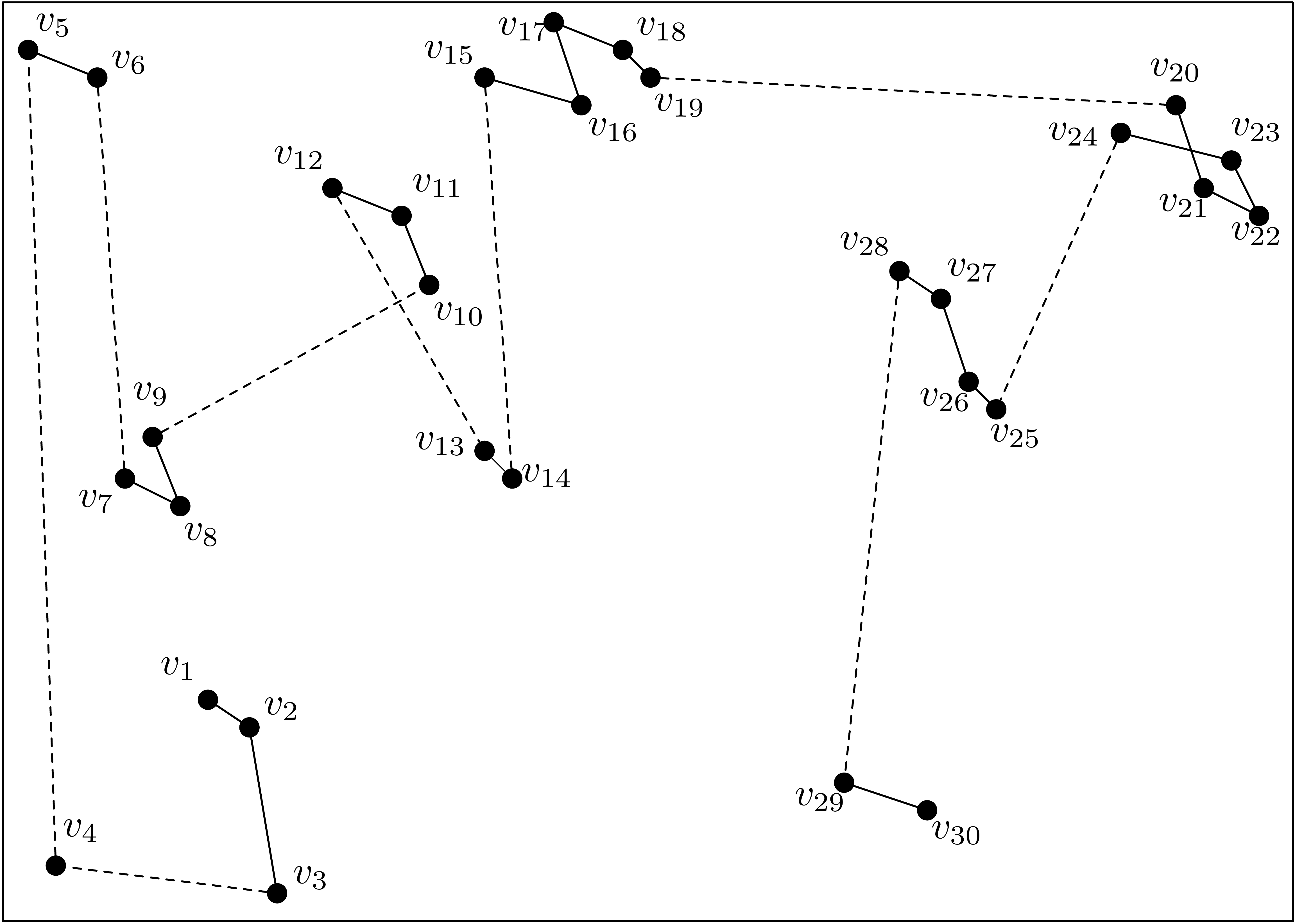}\\
  \caption{The figure shows a robotic roadmap with sensor coverage and
    communication connectivity of a two floors building. Crossing
    edges corresponds to corridors at different floors. A 10-partition
    of the roadmap is here reported. The dashed edges are not traveled
    by any robot.}\label{fig:simulation}
\end{figure}

Given the theoretical nature of this work, in this section we propose
a simulation study to show the effectiveness of our procedures. For
our simulations, we use the Matlab\textregistered \;simulation
environment, and we model the robots as holonomic vehicles of zero
dimension. The communication edges and the motion paths are described
by a given roadmap.

Consider the chain roadmap with $30$ viewpoints in
Fig. \ref{fig:simulation}. Suppose that $10$ robots are assigned to
the patrolling task. In order to obtain a team trajectory with minimum
refresh time, an optimal $10$-partition of the roadmap is computed
(cf. Fig. \ref{fig:simulation}). Additionally, to obtain a minimum
latency team trajectory, each robot is assigned to a different
cluster, its initial position is chosen randomly inside its cluster,
and its velocity is initialized randomly. The motion of each robot is
then determined by Algorithm \ref{algo:synchro}. The resulting team
trajectory is in Fig. \ref{fig:trajectories}, where the team of robots
synchronize on a minimum refresh time and latency team trajectory
after a finite transient.

\begin{figure}
  \centering
  \includegraphics[width=.55\columnwidth]{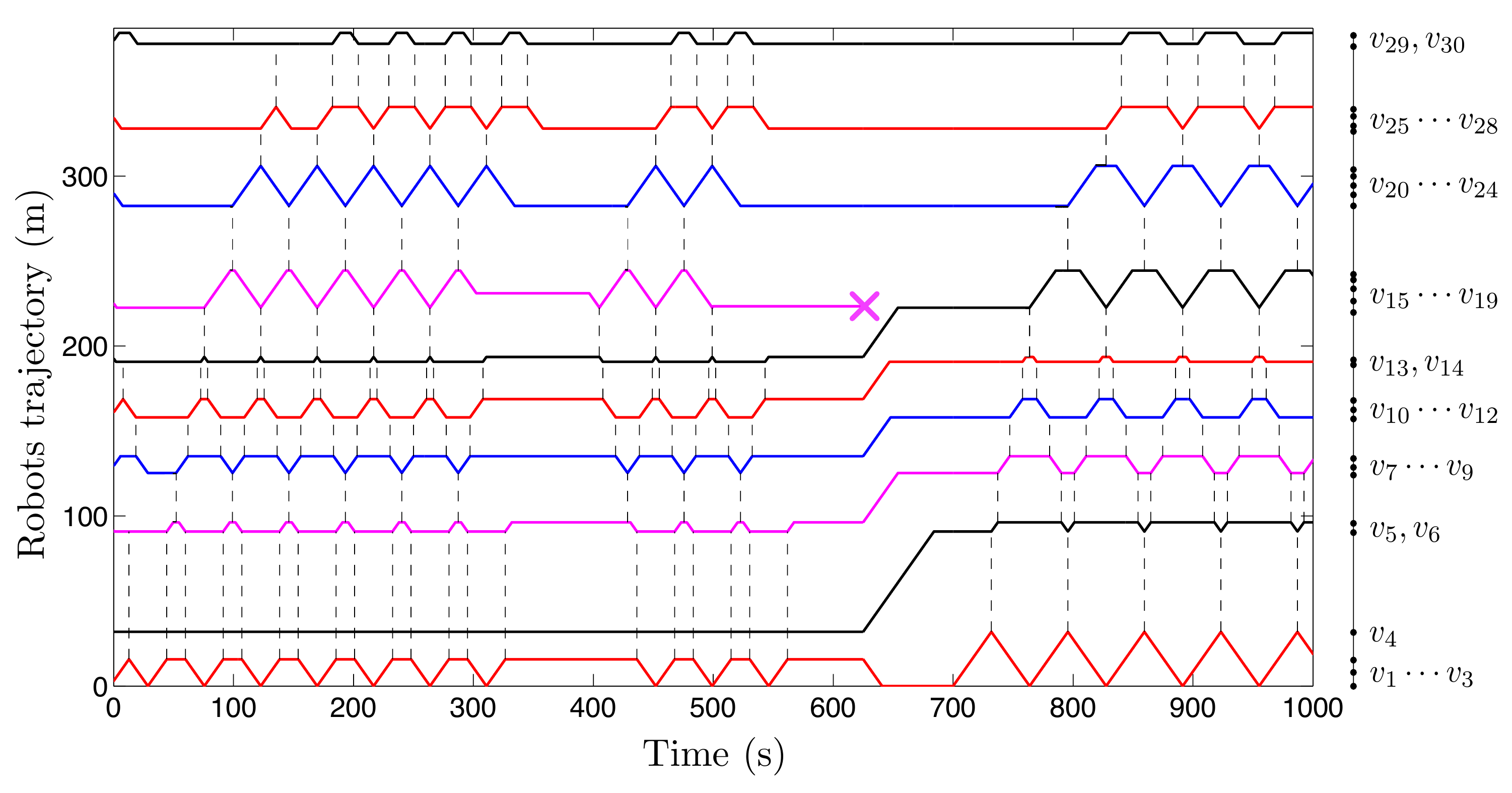}\\
  \caption{For the roadmap of Fig. \ref{fig:simulation}, the team
    trajectory obtained with Algorithm \ref{algo:synchro} is reported
    here. The dashed lines denote a communication among two
    neighboring robots. At time $t=200$ the robots have synchronized
    their motion, and from that moment up to time $t=300$ the team
    trajectory has minimum refresh time and latency. From time $t=300$
    up to time $t=400$, robot $7$ undergoes a temporary failure,
    causing all the other robots to lose synchronization. The team of
    robots synchronizes again when robot $7$ resume its motion. At
    time $t=620$, the failure of robot $7$ is detected by the
    remaining robots, which compute and synchronize on a new
    partition.}\label{fig:trajectories}
\end{figure}

We now test the robustness of our synchronization procedure. As first
case study, we consider a temporary stopping failure, in which a robot
stops for a certain time interval. For instance, suppose that robot
$7$ stops from time $300$ up to time $400$
(cf. Fig. \ref{fig:trajectories}). Notice that, after the failure,
each robot $j$, with $j<7$, gathers at $r_j$, and each robot $k$, with
$k > 7$, gathers at $l_k$ waiting for a communication with the
corresponding neighboring robot. As soon as robot $7$ resumes its
functionalities, the team of robots recover the desired
synchronization. Notice that the transient failure of robot $7$ can be
easily detected by its neighbors by means of a timer mechanism with a
predefined threshold.

As a second case study, we let the robots actuation be affected by
noise, so that the speed of the robots becomes a random variable with
a certain distribution.\footnote{The case in which a robot fails at
  seeing a neighboring robot for a certain interval of time can be
  modeled analogously.}  Precisely, let $\dot x_i=\text{dir}_i + w_i$
be the equation describing the dynamics of robot $i$, where $w_i$ is a
zero mean Gaussian variable with variance $\sigma^2$
$[(\text{m/s})^2]$. We let $\sigma^2 \in \{0, 0.02, \dots, 0.5\}$ and
we run $100$ simulations for each possible value of $\sigma^2$ on the
roadmap of Fig. \ref{fig:simulation}. The refresh time and the latency
of the team trajectories obtained with Algorithm \ref{algo:synchro}
are plotted in Fig. \ref{refresh_error} and in Fig.
\ref{latency_error}, respectively, as a function of $\sigma^2$. Note
that the performance degrade gracefully with the noise magnitude.

\begin{figure}
	\centering
	\subfigure[]{
          \includegraphics[width=0.3\columnwidth]{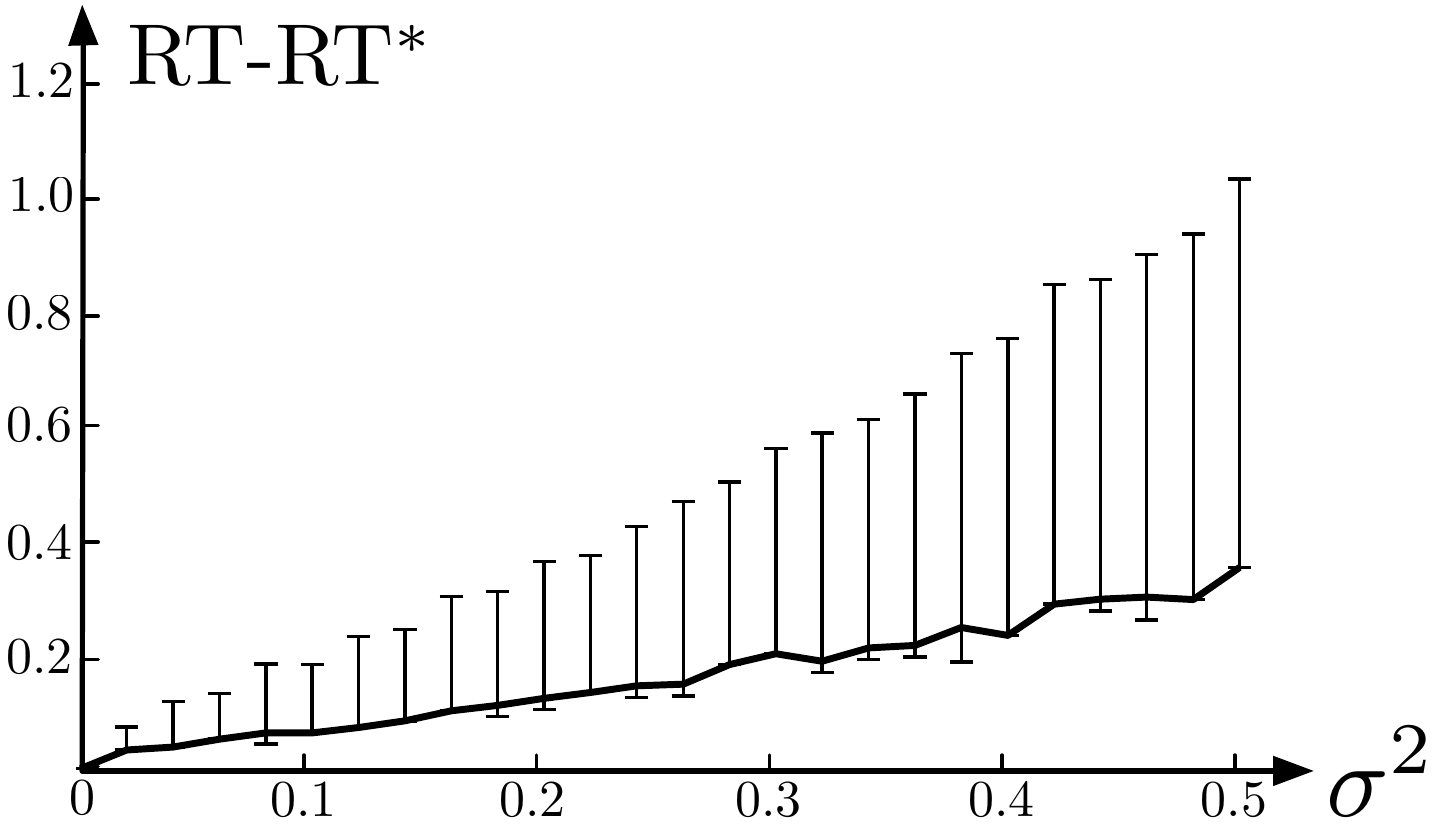}
          \label{refresh_error}}
        \subfigure[]{
          \includegraphics[width=0.3\columnwidth]{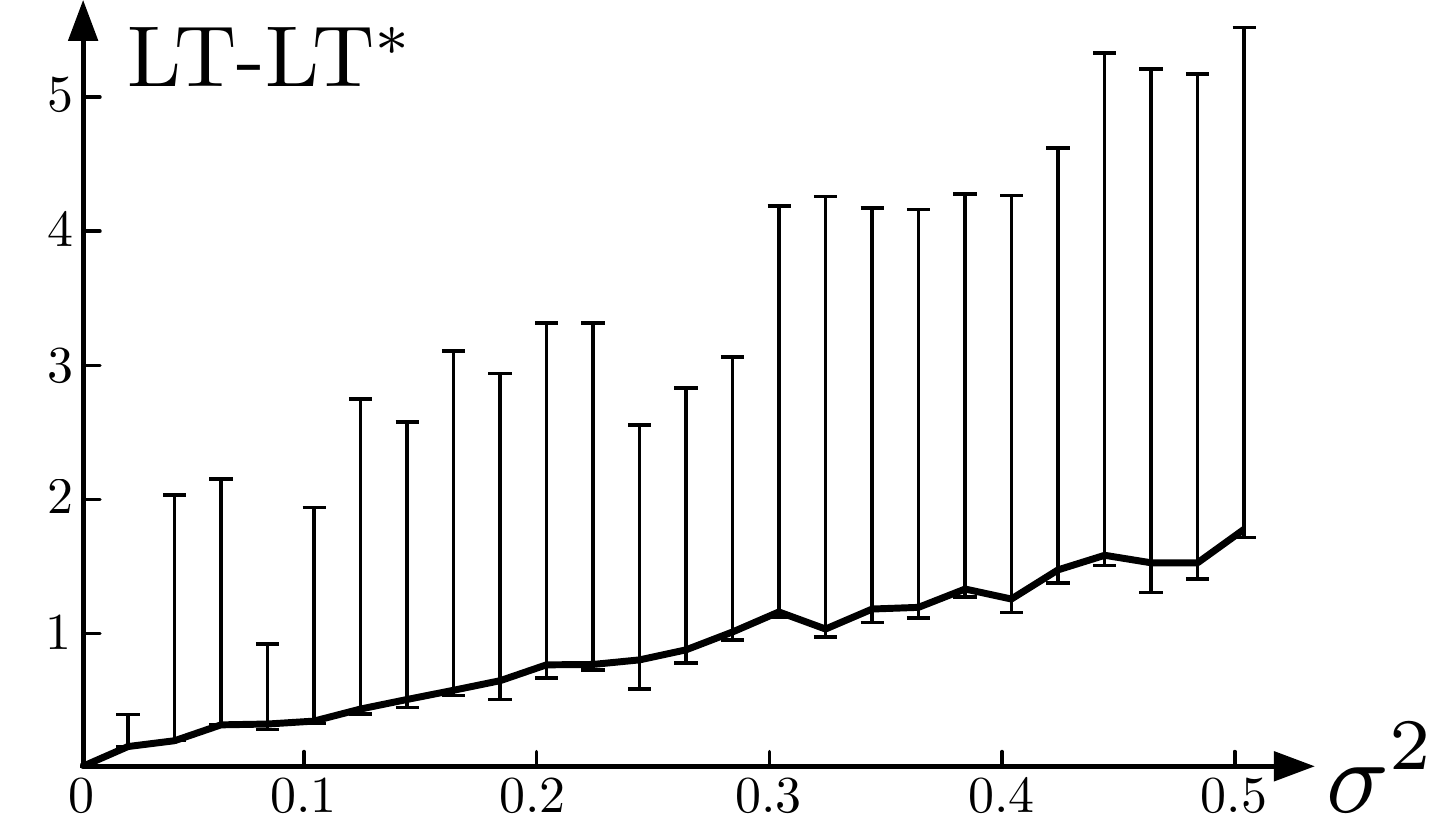}
          \label{latency_error}}
        \caption{For the roadmap in Fig. \ref{fig:simulation}, $100$
          team trajectories have been generated with Algorithm
          \ref{algo:synchro} for each value of $\sigma^2$. The average
          refresh time and the average latency are plotted as
          continuous lines in Fig. \ref{refresh_error} and
          Fig. \ref{latency_error} respectively. The bars indicate the
          minimum and the maximum value of the refresh time and the
          latency, respectively. Notice that the average team
          trajectory performance degrade gracefully with the noise
          variance.}
       \label{fig:noise_sim}
\end{figure}

As third and final case study, we consider the situation in which a
robot definitively stops.\footnote{The case of additional robots
  joining the team is handled in a similar way.} The remaining robots
need to compute a new optimal partition and to synchronize in order to
guarantee an optimal patrolling of the environment. Notice that for
the computation of such a partition by Algorithm \ref{algo:partition}
the chain graph and the number of the robots is required. Suppose that
the failure of the robot $7$ is detected at time $620$ by the
well-behaving robots, and assume that each robot knows the chain
roadmap and the number of operative robots. Algorithm
\ref{algo:partition} and Algorithm \ref{algo:synchro} allow the team
to synchronize on a new team trajectory with minimum refresh time and
latency. Notice that the initial and the final partitions do not
coincide.

\section{Approximation algorithms and heuristics for general
  roadmaps}\label{sec:approx}
The problem of designing minimum refresh time and latency team
trajectories on a chain roadmap has been discussed. In this section
we consider the more general cases of tree and cyclic roadmap, we
characterize the computational complexity of determining optimal
trajectories, and we describe two approximation methods with
performance guarantees. The results we are going to present are
intended for a team of more than one robot. Indeed, if only one robot
is assigned to the patrolling task, then a minimum refresh time
trajectory follows from the computation of the shortest tour through
the viewpoints, for which efficient approximation algorithms already
exist \cite{SA:98}.

\subsection{Minimum refresh time team trajectory on a tree
  roadmap}\label{sec:tree_patrol}
Let $T=(V,E)$ denote an undirected, connected, and acyclic roadmap
(tree). Recall that a vertex path is a sequence of vertices such that
any pair of consecutive vertices in the sequence are adjacent. A tour
is a vertex path in which the start and end vertices coincide, and in
which every vertex of $T$ appears at least once in the sequence. A
depth-first tour of $T$ is a tour that visits the vertices $V$ in a
depth-first order \cite{DP:00}. Let DFT$(T)$ denote the length of a
depth first tour of $T$. Notice that the length of a depth-first tour
of a connected tree equals twice the sum of the length of the edges of
the tree, and that any depth-first tour is a shortest tour visiting
all the vertices. We now show that, for the case of tree roadmap, the
set of cyclic-based and partition-based trajectories described in
\cite{YC:04} does not contain, in general, a minimum refresh time
trajectory. Recall that in a cyclic-based strategy the robots travel
at maximum speed and equally spaced along a minimum length tour
visiting all the viewpoints. Consider the tree roadmap of
Fig. \ref{tree_1}, and suppose that two robots are assigned to the
patrolling task. Clearly, the minimum refresh time is $2\varepsilon$,
while the refresh time of a cyclic strategy equals
$1+\varepsilon$. Consider now the tree roadmap in
Fig. \ref{tree_equal}, where the edges have unit length, and assume
that two robots are in charge of the patrolling task. Observe that any
partition of cardinality $2$ contains a chain of length $2$, so that,
since only one robot is assigned to each cluster, the minimum refresh
time that can be obtained is $4$. Suppose, instead, that the robots
visit the vertices of the roadmap as specified in Table
\ref{table:visiting_order}, where $x(t)$ denotes the position of a
robot at time $t$. Since the refresh time of the proposed trajectory
is $3$, we conclude that neither the cyclic-based nor the
partition-based strategy may lead to a minimum refresh time team
trajectory on a tree roadmap.

\begin{table}[ht]
  \caption{}
  \centering
\begin{tabular}{ccccccccc}
\hline\hline
Robot & $x(0)$ & $x(1)$ & $x(2)$ & $x(3)$ & $x(4)$ & $x(5)$ & $x(6)$ &
$x(7)$\\
\hline
1 & $v_1$ & $v_2$ & $v_4$ & $v_2$ & $v_3$ & $v_2$ & $v_1$ & $\cdots$\\
2 & $v_2$ & $v_3$ & $v_2$ & $v_1$ & $v_2$ & $v_4$ & $v_2$ & $\cdots$\\
\hline
\end{tabular}
\label{table:visiting_order}
\end{table}

\begin{figure}
	\centering
	\subfigure[]{
          \includegraphics[width=0.15\columnwidth]{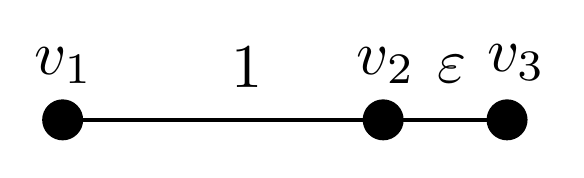}
          \label{tree_1}}
        \subfigure[]{
          \includegraphics[width=0.26\columnwidth]{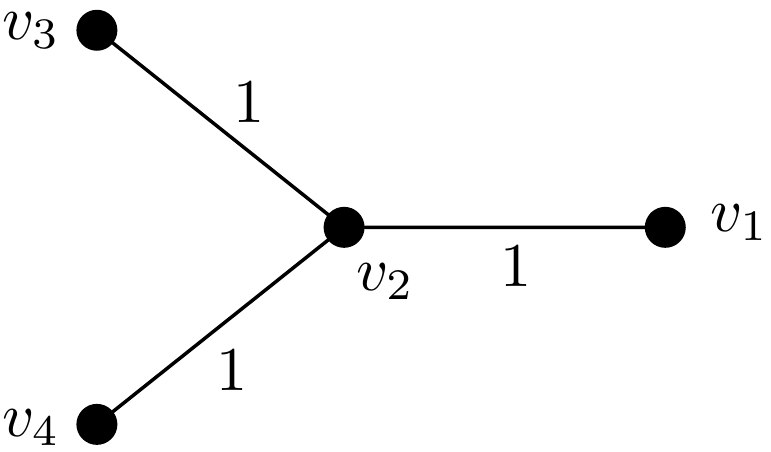}
          \label{tree_equal}}
        \caption{Two examples of tree roadmap. For the roadmap in
          Fig. \ref{tree_1} the family of cyclic based trajectories
          does not contain a minimum refresh time team
          trajectory. Instead, for the roadmap in Fig. \ref{tree_equal}
          the family of partitioned based trajectories does not
          contain a minimum refresh time team trajectory.}
\end{figure}

We now introduce some definitions. Let $X$ be a team trajectory on the
tree roadmap $T$.
We say that the edge $(v_j,v_z)\in E$ is \emph{used} by $X$ if there
exists $i\in \{1,\dots,m\}$ and $(t_1,t_2) \in [0, \text{RT}(X)]^2$ such
that $x_i(t_1) = v_j$ and $x_i(t_2) = v_z$, and it is \emph{unused}
otherwise. Note that, because in a tree there exists only one path
connecting two vertices, the above condition ensures that the edge
$(j,z)$ is traveled by the robot $i$. Let $\bar E$ denote the set of
unused edges, and let $F_T$ be the forest obtained from $T$ by
removing the edges $\bar E$ from $E$, i.e., the collection of
vertex-disjoint subtrees $\{T_1,\dots,T_k\}$, with $T_i=(V_i,E_i)$,
such that $V= \cup_{i=1}^k V_i$ and $E_i \subseteq E$, for each
$i\in\{1,\dots,k\}$. Let $m_i$ be the number of robots that visit at
least one vertex of $T_i$ in the interval $[0,\text{RT(X)}]$, and note
that $m_i > 0$ in a finite refresh time trajectory. Let
$M=\{m_1,\dots,m_k\}$. 
Notice that the same subtree collection can be associated with
different team trajectories. We say that a team trajectory is
\emph{efficient} if its refresh time is the smallest among all the
team trajectories associated with the same subtree collection.

\begin{theorem}[Efficient team trajectory]\label{eff_team}
  Let $(F_T,M)$ be the subtree collection associated with the team
  trajectory $X$ on the tree roadmap $T$, where
  $F_T=\{T_1,\dots,T_k\}$, and $M=\{m_1,\dots,m_k\}$. Then, $X$ is
  efficient if
    \begin{align*}
      \textup{RT}(X)=\max_{j \in \{1,\dots,k\}}
      \operatorname{DFT}(T_j)/m_j.
    \end{align*}
\end{theorem}
\begin{proof}
  Let $i\in \{1,\dots,k\}$, and let $m_i$ be the number of robots
  assigned to $T_i$. Notice that the robots in $T_i$ travel, in total,
  at least $\text{DFT}(T_i)$ to visit all the vertices. Since the
  speed of the robots is bounded by $1$, the smallest refresh time for
  the vertices of $T_i$ is $\text{DFT}(T_i)/m_i$.
\end{proof}

An efficient team trajectory, can be computed with the following
procedure. See Table \ref{table:visiting_order} for an example.
\begin{lemma}[Efficient team trajectory computation]\label{lemma_team}
  Let $(F_T,M)$ be a subtree collection of a tree roadmap, where
  $F_T=\{T_1,\dots,T_k\}$, and $M=\{m_1,\dots,m_k\}$. An efficient
  team trajectory is as follows: for each $i\in\{1,\dots,k\}$,
  \begin{enumerate}
    \item compute a depth-first tour $\tau_i$ of $T_i$,
    \item equally space $m_i$ robots along $\tau_i$, and
    \item move the robots clockwise at maximum speed on $\tau_i$.
   \end{enumerate}
\end{lemma}
\begin{proof}
  Since every vertex of $T_i\in F_T$ appears at least once in a depth
  first tour $\tau_i$ of $T_i$, and the robots move with maximum speed
  and equally spaced along $\tau_i$, every vertex is visited at most
  every $\text{DFT}(T_i) /m_i$. The statement follows.
\end{proof}

Let $P(m)$ be the partition set of $m$, i.e., the set of all the
sequences of integers whose sum is $m$. The following problem is
useful to characterize the complexity of designing minimum refresh
time trajectories on a tree roadmap.


\begin{problem}[Optimal subtree collection]\label{prob:optimal_forest}
  Let $T$ be a tree roadmap and let $m$ be the number of robots. Find
  a subtree collection $(F_T,M)$ that minimizes $\max_{j\in
    \{1,\dots,|F_T|\} } \operatorname{DFT}(T_j)/m_j$ subject to $M\in
  P(m)$ and $|F_T|=|M|$.
\end{problem}

\begin{lemma}[Equivalent problem]\label{equiv_prob}
  For the case of a tree roadmap, the Team refresh time problem and
  the Optimal subtree collection problem are equivalent.
\end{lemma}
\begin{proof}
  As a consequence of Theorem \ref{eff_team}, the minimum refresh time
  on a tree roadmap $T$ can be written as $\min_{(F_T,M)} \max_{j \in
    \{1,\dots,k\}} \text{DFT}(T_j)/m_j$, where $(F,M)$ is a subtree
  collection of $T$, and $|M|=|F_T|=k\le m$. It follows that a
  solution to Problem \ref{prob:minRT} can be derived in polynomial
  time from a solution to Problem \ref{prob:optimal_forest} by using
  the procedure described in Lemma \ref{lemma_team}. Suppose now we
  have a solution to Problem \ref{prob:minRT}. Then an optimal subtree
  collection follows from the identification of the unused edges. We
  conclude that the two optimization problems are equivalent.
\end{proof}

We now state our main result on the design of minimum refresh time
team trajectory on a tree roadmap.

\begin{theorem}[Computing a minimum refresh time team trajectory
    on a tree]\label{complexity}
  Let $T$ be a tree roadmap with $n$ vertices, and let $m$ be the
  number of robots. A minimum refresh time team trajectory on $T$ can
  be computed in $O((m-1)! n)$ time.
\end{theorem}
\begin{proof}
  Recall from \cite{HN-KO:04} that an optimal subtree collection can
  be computed in $O((m-1)! n)$. Then, by using Lemma \ref{equiv_prob}
  and Lemma \ref{lemma_team}, the claimed statement follows.
\end{proof}

As a consequence of Theorem \ref{complexity}, the problem of designing
minimum refresh time team trajectories on a tree roadmap is
computationally \emph{easy} for any finite number of robots. In our
design procedure, we first compute an optimal subtree collection of
the given tree, and then we schedule the robots trajectory according
to Lemma \ref{lemma_team}.

\subsection{Minimum refresh time team trajectory on a cyclic
  roadmap}\label{sec:heurist}
In this section we propose two approximation methods for the Team
refresh time problem in the case of a cyclic, i.e., not acyclic,
roadmap. These solutions are obtained from a transformation of the
cyclic roadmap into an acyclic roadmap.

Let $G=(V,E)$, with $|V|=n$, be an undirected and connected
roadmap. Note that there exists an open tour $\tau$ with at most
$2n-4$ edges that visits all the vertices.\footnote{An open tour with
  at most $2n-4$ edges that visits all the vertices can be constructed
  starting from a leaf of a spanning tree of $G$.}
\begin{figure}[tb]
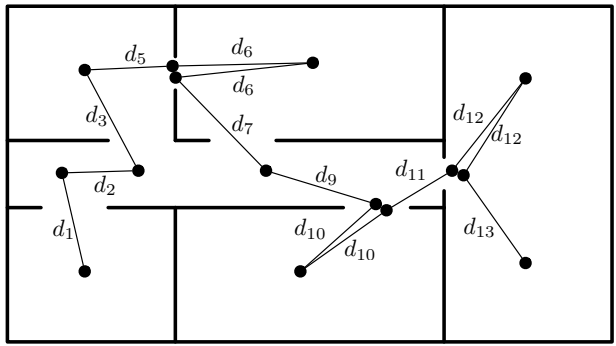

  \figCommChain
  \caption{The figure shows a chain roadmap associated with the cyclic
    roadmap of Fig. \ref{fig:CommGraph}. Notice that the cycles in
    Fig. \ref{fig:CommGraph} have been broken. Moreover, $3$ vertices
    ($3$ edges) of Fig. \ref{fig:CommGraph} are repeated twice in the
    chain.}\label{fig:treechain}
\end{figure}
We construct a chain roadmap $\Gamma$ from $\tau$ by doubling its
repeated vertices and edges, so that $\Gamma$ has at most $2n-3$
vertices and at most $2n-4$ edges, and such that the length of the
$i$-th edge of $\Gamma$ equals the length of the $i$-th edge of $\tau$
(cf. Fig. \ref{fig:treechain}). Our first approximation method
consists of applying Algorithm \ref{algo:synchro} to an optimal
$m$-partition of $\Gamma$.
\begin{theorem}[Performance ratio]\label{thm:perf_ratio}
  Let $G$ be a connected roadmap, let $n$ be the number of vertices of
  $G$, and let $\gamma$ be ratio of the longest to the shortest length
  of the edges of $G$. Let $\text{RT}^*$ be the minimum refresh time
  on $G$. Let $\tau$ be an open tour with $2n-4$ edges that visits all
  the $n$ vertices, and let $\Gamma$ be the chain roadmap associated
  with $\tau$. Let $\text{RT}_{\Gamma}^*$ be the minimum refresh time
  on $\Gamma$. Then
  \begin{align*}
    \text{RT}_{\Gamma}^* \le \left( \frac{n-2}{n} \right) 8 \gamma \text{RT}^*.
  \end{align*}
\end{theorem}
\begin{proof}
  Let $\underline w$ be the shortest length of the edges of $G$, and
  note that the length of $\Gamma$ is upper bounded by
  $2(n-2)\gamma\underline w$. It follows that $\text{RT}_{\Gamma}^*
  \le \frac{4 (n-2)\gamma \underline w}{m}$. Since $m < n$ by
  assumption, some robots need to move along $G$ for all the
  viewpoints to be visited. Because each robot can visit only a vertex
  at a time, at least $\left\lceil \frac{n}{m} - 1 \right\rceil$ steps
  are needed to visit all the vertices of $G$, and therefore
  $\text{RT}^* \ge \left\lceil \frac{n}{m} - 1 \right\rceil \underline
  w \ge \frac{1}{2}\frac{n}{m} \underline w$.
  By taking the ratio of the two quantities the statement follows.
\end{proof}

It should be noticed that, when $\gamma$ grows, the performance of our
procedure might degrade. For instance, suppose that the roadmap is as
in Fig. \ref{fig:ratio_inf}, and suppose that $4$ robots are assigned
to the patrolling task. As long as $\varepsilon<1$, a minimum refresh
time strategy requires one robot to patrol the viewpoints
$\{v_1,v_2\}$, while the second, third, and fourth robot stay on the
viewpoints $v_3$, $v_4$, and $v_5$ respectively. It follows that
$\text{RT}^*=2\varepsilon$. On the other hand, an optimal
$m$-partition of any chain graph associated with a tour that visits
all the viewpoints has dimension at least $1$. Consequently, the
refresh time of the team trajectory obtained with Algorithm
\ref{algo:synchro} equals $2$, and the ratio $\text{RT}_{\Gamma}^* /
\text{RT}^*$ grows proportionally to $\varepsilon^{-1}$.
\begin{figure}[tb!]
  \centering
  \includegraphics[width=0.6\columnwidth]{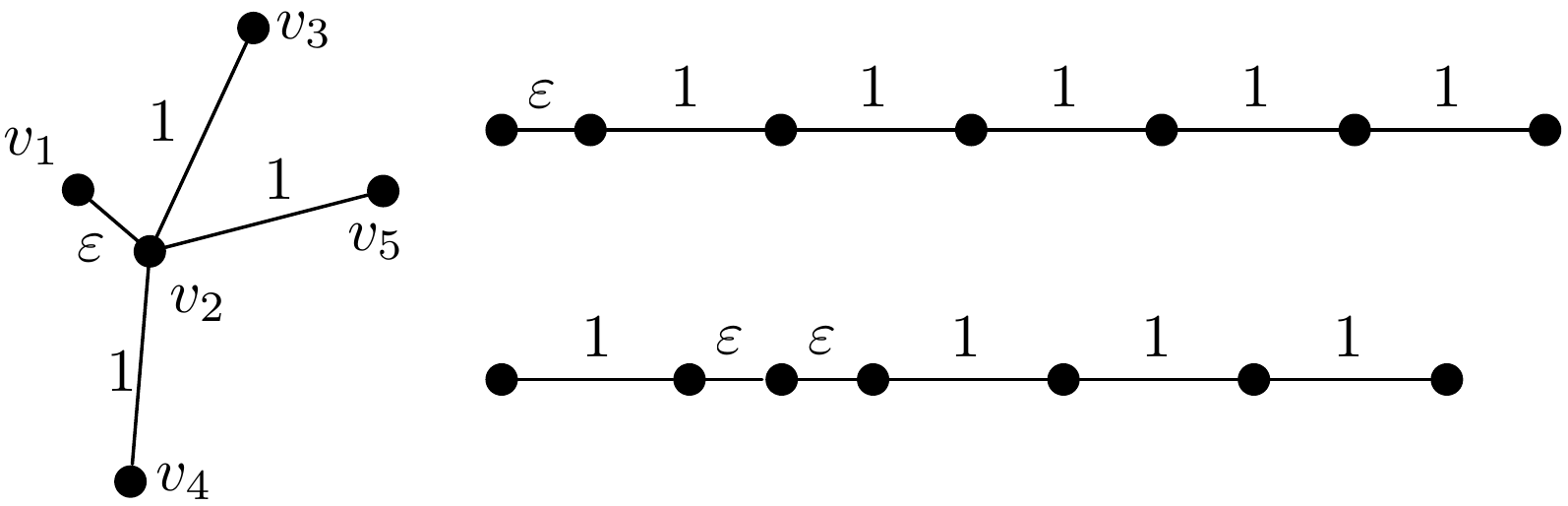}\\
  \caption{A tree roadmap and two corresponding chain roadmaps. If the
    number of robots is $4$, then the performance ratio
    $\text{RT}_{\Gamma}^* / \text{RT}^*$ grows with
    $\varepsilon^{-1}$.}\label{fig:ratio_inf}
\end{figure}

We next describe a polynomial time constant factor approximation
algorithm for the minimum refresh time problem. Given a roadmap
$G=(V,E)$ and a positive integer $k<|V|$, we define a path cover of
cardinality $k$ as the collection of paths $\{p_1,\dots,p_k\}$ such
that $V \subseteq \bigcup_{i=1}^k p_i$. Let the cost of a path equal
the sum of the lengths of its edges. The min-max path cover problem
asks for a minimum cost path cover for the input graph, where the cost
of a cover equals the maximum cost of a path in the cover. The
following result is known.
\begin{theorem}[Min-max path cover
  \cite{EMA-RH-AL:06}]\label{result_path_cover}
  There exists a $4$-approximation polynomial algorithm for the
  \textit{NP-hard} min-max path cover problem.
\end{theorem}

Following Theorem \ref{result_path_cover}, given a graph $G$, there
exists a polynomial time algorithm that computes a path cover of $G$
with cost at most $4$ times greater than the cost of any path cover of
$G$. We now state our approximation result for the \textit{NP-hard}
Team refresh time problem.
\begin{lemma}[$8$-approximation refresh time]\label{4_approx}
  There exists an $8$-approximation polynomial algorithm for the
  \textit{NP-hard} Team refresh time problem.
\end{lemma}
\begin{proof}
  Let $\{p_1,\dots,p_m\}$ be a $4$-approximation path cover of the
  graph $G$. Note that the length of each path is within
  $4\text{RT}^*$. Indeed, in a minimum refresh time team trajectory
  starting at time $0$ and with unitary velocity, every vertex is
  visited within time $\text{RT}^*$. Let $X$ be a team trajectory
  obtained by letting robot $i$ sweep at maximum speed the path
  $p_i$. Clearly, $\text{RT}(X) \le 8 \text{RT}^*$. Because of Theorem
  \ref{result_path_cover}, the team trajectory $X$ can be computed in
  polynomial time.
\end{proof}

\begin{figure}
    \centering
    \includegraphics[width=.5\columnwidth]{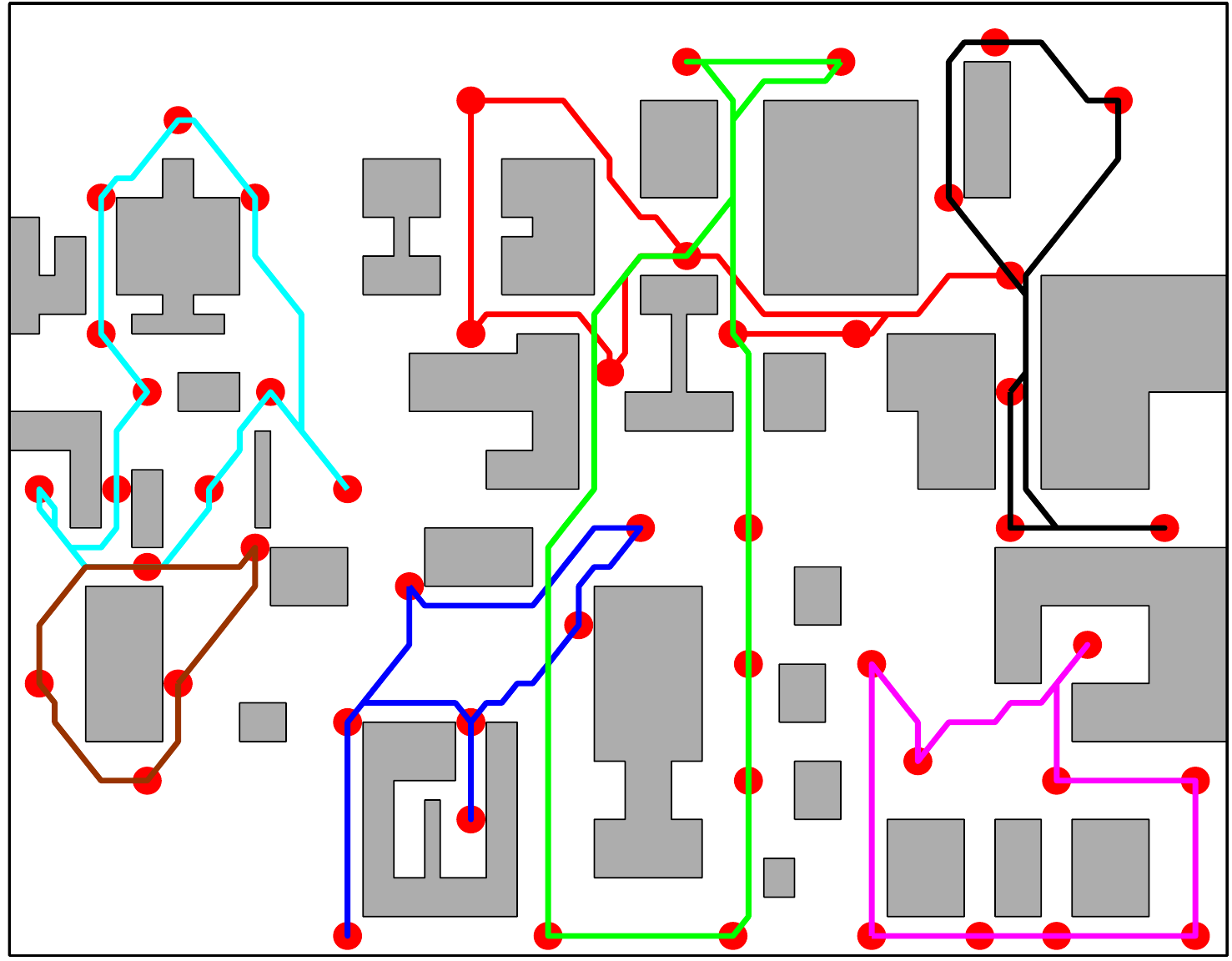}
    \caption{The picture shows the trajectories of $7$ robots for the
      patrolling of a part of the UCSB campus. The viewpoints (red
      circles) are chosen so as to provide sensor coverage of the
      whole area. For the design of the patrolling trajectory, a
      roadmap is first constructed as described in Section
      \ref{sec:prob_set}. Then, a path covering of cardinality $7$ is
      computed with the procedure in \cite{EMA-RH-AL:06}, and each
      robot is assigned a different path. Finally, the trajectory of
      each robot consists in sweeping at maximum speed the tour
      obtained by shortcutting the assigned path. The refresh time of
      the proposed team trajectory is proven to be within a factor of
      $8$ of the minimum refresh time for the given roadmap.}
    \label{fig:path_cover}
\end{figure}

Following Lemma \ref{4_approx}, for any given roadmap and any number
of robots, a team trajectory with refresh time within a factor of $8$
of the optimal refresh time can be constructed by computing a path
covering of the roadmap, and by assigning a different path to each
robot. An example is in Fig. \ref{fig:path_cover}, a movie of which is
included in the multimedia material.

\begin{remark}[Improving the team trajectory]
  Several existing heuristics can be used to improve upon the
  trajectories in Fig. \ref{fig:path_cover}. For instance, since the
  robots move in a metric space, shortcutting techniques may be
  applied \cite{VVVa:01}. Because these heuristics do not guarantee an
  improvement of the optimality gap of our trajectories, they are not
  considered in this work, and they are left as the subject of future
  investigation.
\end{remark}



\section{Conclusion and future work}\label{sec:future_work}
The design of team trajectories to cooperatively patrol an environment
has been discussed. After defining the problem and the performance
criteria, we analyze the computational complexity of the design
problem as a function of the shape of the environment to be
patrolled. For the case of a chain environment, we describe a
polynomial algorithm to compute minimum refresh time and latency team
trajectories. For the case of a tree environment, under the technical
assumption of a constant number of robots, we identify a polynomial
time algorithm to compute minimum refresh time team
trajectories. Finally, the general case of cyclic environment is shown
to be \textit{NP-hard}, and two approximation algorithms with
performance guarantees have been proposed.

Interesting aspects requiring further investigation include a
throughout study of the latency optimization problem for cyclic
roadmaps, the development of more efficient approximation algorithms,
and the introduction of a more general communication framework, in
which the robots are allowed to communicate while traveling the edges
of the roadmap. The study of average performance criteria and the
extension to dynamically changing environments are also of
interest. Finally, an hardware implementation of our algorithms would
further strengthen our theoretical findings.

\bibliographystyle{IEEEtran}

\end{document}